\newif\ifPAPER
\PAPERfalse 

\newif\ifTR 
\TRtrue

\ifPAPER
\PassOptionsToPackage{numbers}{natbib}
\RequirePackage{fix-cm}
\documentclass[smallextended,natbib]{svjour3}       		
\bibpunct{[}{]}{,}{n}{}{,} 									
\smartqed  									 				
\makeatletter \let\cl@chapter\relax \makeatother			
\fi

\ifTR
\documentclass[11pt]{article}
\usepackage{fullpage}
\usepackage{authblk}
\usepackage[numbers]{natbib}					
\fi

\usepackage[dvipsnames]{xcolor}					
\usepackage{color}								
\usepackage[toc,title]{appendix}				

\usepackage{comment}							
\usepackage{pdfpages}                           
\usepackage{xspace}								
\usepackage{epstopdf} 							

\usepackage{graphicx}							
\usepackage{multirow}							
\usepackage{paralist}							
\usepackage{verbatim}							
\usepackage{longtable}							
\usepackage{booktabs}							
\usepackage{subcaption}							

\makeatletter
\def\BState{\State\hskip-\ALG@thistlm}
\makeatother

\usepackage{tikz}								
\usetikzlibrary{graphs,graphs.standard,quotes}	
\usetikzlibrary{backgrounds}					
\usetikzlibrary{fit} 							

\usepackage[T1]{fontenc}
\usepackage{lmodern}
\usepackage[utf8]{inputenc}						
\usepackage{soul}								
\usepackage[normalem]{ulem}						

\usepackage[linesnumbered,noend]{algorithm2e}   
\usepackage{setspace}
\DontPrintSemicolon
\SetKw{KwAnd}{and}
\SetKw{KwOr}{or}
\SetKw{KwTrue}{true}
\SetKw{KwFalse}{false}

\usepackage{forloop}							
\usepackage{ifthen}								

\colorlet{darkred}{red!80!black}
\colorlet{darkgreen}{green!60!black}
\colorlet{darkblue}{blue!80!black}
\colorlet{darkorange}{orange!70!black}
\definecolor{Green}{cmyk}{1, 0.2, 0.4, 0.1}
\definecolor{Orange}{cmyk}{0, 0.61, 0.87, 0}
\definecolor{Purple}{rgb}{0.75, 0.0, 1.0}
\definecolor{purple}{rgb}{0.5,0,0.5}
\definecolor{dgreen}{rgb}{0.1,0.9,0.5}
\definecolor{gabysgreen}{cmyk}{0.80, 0.1, 0.90, 0}

\usepackage{amssymb}							
\usepackage[fleqn]{amsmath}						
\usepackage{amsthm}								
\usepackage{bm}									
\usepackage{sansmath}
\usepackage{mathtools}							
\usepackage{nicefrac}							
\usepackage{siunitx}							

\usepackage{hyperref}							
\hypersetup{
	colorlinks,
	linkcolor={blue!50!black},
	citecolor={green!50!black},
	urlcolor={blue!80!black}
}
\usepackage[capitalise,noabbrev]{cleveref}		

\newcommand{\irg}[2]{[#1\!:\!#2]}
\newcommand{\R}{\mathbb{R}}						
\newcommand{\N}{\mathbb{N}}						
\newcommand{\x}{\vc{x}}							
\renewcommand{\u}{\vc{u}}						
\newcommand{\lrbr}[1]{\left\lbrace #1 \right\rbrace} 

\newcommand\ga{\alpha}
\newcommand\gga{\bm{\alpha}}

\newcommand\ggb{\bm{\beta}}
\newcommand\gd{\delta}
\newcommand\ggd{\bm{\delta}}
\newcommand\gc{\gamma}
\newcommand\ggc{\bm{\gamma}}

\newcommand\gl{\lambda}
\newcommand\ggl{\bm{\lambda}}

\newcommand\mmu{\bm{\mu}}

\def\x{\mathbf x}
\def\aa{\mathbf a}
\def\y{\mathbf y}
\def\z{\mathbf z}
\def\s{\mathbf s}
\def\t{\mathbf t}
\def\v{\mathbf v}
\def\w{\mathbf w}

\def\f{\mathbf f}
\def\oo{\mathbf o}
\def\e{\mathbf e}
\def\u{\mathbf u}
\def\v{\mathbf v}
\def\b{\mathbf b}
\def\cc{\mathbf c}
\def\d{\mathbf d}

\def\g{\mathbf g}

\def\q{{\mathbf q}}
\def\r{{\mathbf r}}

\def\AA{{\mathcal A}}

\newcommand\II{{\mathcal I}}
\newcommand\JJ{{\mathcal J}}
\newcommand\KK{{\mathcal K}}

\newcommand\MM{{\mathcal M}}

\newcommand\RR{{\mathcal R}}

\newcommand\UU{{\mathcal U}}
\newcommand\VV{{\mathcal V}}

\newcommand\XX{{\mathcal X}}

\newcommand\Ab{{\mathsf{A}}}
\newcommand\Bb{{\mathsf B}}
\newcommand\Cb{{\mathsf C}}

\newcommand\Eb{{\mathsf E}}
\newcommand\Fb{{\mathsf F}}
\newcommand\Gb{{\mathsf G}}
\newcommand\Hb{{\mathsf H}}
\newcommand\Ib{{\mathsf I}}

\newcommand\Ob{{\mathsf O}}

\newcommand\Xb{{\mathsf X}}

\newcommand{\T}{^\mathsf{T}}	
\def\beq#1{$$}														

\def\ol#1{{\overline #1}}

\def\eps{\varepsilon}

\def\ol#1{{\overline #1}}

\def\bea#1{\begin{array}{#1}}
	\def\ea{\end{array}}
\def\ignore#1{}

\theoremstyle{plain}
\newtheorem{thm}{Theorem} 
\newtheorem{exmp}{Example} 
\newtheorem{asm}{Assumption}

\newtheorem{lem}[thm]{Lemma}

\newtheorem{cor}{Corollary} 
\newtheorem{rem}{Remark}
\theoremstyle{definition}
\theoremstyle{remark}

\allowdisplaybreaks

\usepackage{titlesec}
\titlespacing\section{0pt}{4pt plus 2pt minus 2pt}{0pt plus 1pt minus 3pt}
\titlespacing\subsection{0pt}{3pt plus 2pt minus 2pt}{0pt plus 1pt minus 3pt}
\titlespacing\subsubsection{0pt}{2pt plus 2pt minus 2pt}{0pt plus 1pt minus 3pt}

\ifTR
\usepackage{fancyhdr}
\fi
\begin{document}

	\ifTR
	\title{Uncertainty Preferences in Robust Mixed-Integer \\ Linear Optimization with Endogenous Uncertainty \thanks{Supported by the
			Austrian Science Fund (project DK W 1260-N35).}}
	
	\author{Immanuel Bomze, Markus Gabl} 
	\affil{ISOR/VCOR/VGSCO and ds:univie, University of Vienna, Austria. \texttt{[immanuel.bomze|markus.gabl]@univie.ac.at}\\
		\texttt{+43-1-4277-[38652|838672]}
	}
	\renewcommand\Authands{ and }
	\maketitle
	\else
	\fi	
	\begin{abstract}
		In robust optimization one seeks to make a decision under uncertainty, where the goal is to find the solution with the best worst-case performance. The set of possible realizations of the uncertain data is described by a so-called uncertainty set. In many scenarios, a decision maker may influence the uncertainty regime they are facing, for example, by investing in market research, or in machines which work with higher precision. Recently, this situation was addressed in the literature by introducing decision dependent uncertainty sets (endogenous uncertainty), i.e., uncertainty sets whose structure depends on (typically discrete) decision variables. In this way, one can model the trade-off between reducing the cost of robustness versus the cost of the investment necessary for influencing the uncertainty. However, there is another trade-off to be made here. With different uncertainty regimes, not only do the worst-case optimal solutions vary, but also other aspects of those solutions such as max-regret, best-case performance or predictability of the performance. A decision maker may still be interested in having a performance guarantee, but at the same time be willing to forgo superior worst-case performance if those other aspects can be enhanced by switching to a suitable uncertainty regime.  We introduce the notion of \textit{uncertainty preference} in order to capture such stances. We introduce a multi-objective optimization based and a bilevel optimization based model that integrate these preferences in a meaningful way. Further, we present three ways to formalize uncertainty preferences and study  the resulting mathematical models. The goal is to have reformulations/approximations of these models which can be solved with standard methods. The workhorse is mixed-integer linear and conic optimization. We apply our framework to the uncertain shortest path problem and conduct numerical experiments for the resulting models. We can demonstrate that our models can be handled very well by standard mixed-integer linear solvers.
		\ifTR
		\\
		\\
		\textbf{Keywords:} Robust optimization $\cdot$ Discrete Optimization $\cdot$ Endogenous Uncertainty $\cdot$ Risk Aversion
		\fi
	\end{abstract}	
	\pagebreak	
	\section{Introduction}
	\subsection{Background and basic ideas}	
	The presence of uncertainty in decision problems has been a driving factor of research in optimization. Sources for uncertainty are numerous: relevant data may not be available at the time of decision making and have to be estimated. Even if available, relevant data may be riddled with measurement errors, and decisions made may not be implementable exactly. Poor understanding, or ignorance, of uncertainty in decision making may lead to bad performance of a decision.  The resulting challenges for optimization have been met under two major paradigms, namely robust optimization and stochastic optimization. The latter is concerned with optimizing moments (mostly expected values) of functions that depend on random variables under a known probability law  (see e.g. \cite{kall1994stochastic,shapiro2014lectures}). This approach is appealing as it uses the maximum information about an uncertain process, i.e. its probability distribution. Unfortunately, this comes at the cost of having to deal with mathematical optimization models that more often than not are intractable, so that their solutions have to be approximated with computationally expensive methods. Also, the probability distribution in question may be unknown so that a new source for uncertainty arises, which itself poses a major challenge.
	
	The alternative paradigm of dealing with uncertainty is that of robust optimization. Here, uncertain data is assumed to reside within an appropriately structured uncertainty set and one seeks to find the solution with the best worst-case performance among those solutions that are feasible for all possible realizations of the uncertain data (see e.g. \cite{ben2009robust,gorissen2015practical}). This approach requires more modest assumptions, as no probability law has to be assumed. In addition, the robust counterpart of an uncertain problem can often be reformulated as a convex optimization problem that can be solved in polynomial time. The major drawback of this approach is, that the solutions obtained are often too conservative, courtesy of the pessimistic perspective inherent in robust optimization. Many attempts have been made to alleviate this shortcoming which lead to the development of concepts such as adjustable robustness (see \cite{ben2004adjustable}), max-regret robustness (see \cite{kouvelis1997robust}), among others.
	
	A recent line of literature introduced the notion of endogenous uncertainty where the decision maker has control over an additional set of (often binary) variables which govern the structure of the uncertainty. This approach allows the modelling of situations where the uncertainty can be influenced by the decision maker by investing, for instance, in market research or more precise machines.
	The concept has been known in the stochastic optimization community for some time, see  \cite{jonsbraaten1998class} for an early example. Other examples can be found in \cite{goel2006class}, who applied a stochastic optimization framework with decision dependent uncertainty to capacity expansion of process networks and the sizes problem in production, and in \cite{green2013nursevendor}, who considered personnel staffing in the presence of endogenous absenteeism. More recently, the concept of endogenous uncertainty was adopted for distributionally robust optimization, see e.g.  \cite{luo2020distributionally,noyan2018distributionally,ryu2019nurse}. However, in this text we will focus entirely on the robust optimization point of view.

	From a robust optimization perspective, an investment  in  a modified uncertainty set is only acceptable if the expenses do not exceed the gains in worst-case pay-off under the new uncertainty set. The arising decision problem was recently addressed in \cite{lappas2018robust,nohadani2018optimization}, where the authors studied reformulations  amenable to techniques for solving a mixed-integer linear optimization problem (MILP) or more generally, a conic mixed-integer linear optimization problem (C-MILP). In \cite{lappas2018robust} it is also proved that the class of robust optimization problems under endogenous uncertainty is NP-hard. We will expand on the idea of endogenous uncertainty by introducing the notion of uncertainty preference, aiming at tractable formulations.

	\subsection{Uncertainty preferences and main motivation of this study}\label{sec:Uncertainty preferences and main motivation of this study}
	The framework of robust optimization with endogenous uncertainty allows us to model the trade-off between the cost of uncertainty reduction and worst-case performance.
	However, there is another trade-off to be made in such a scenario. Since we influence the structure of the uncertainty by our decision we do not only change the robustly optimal solution, we also affect other qualities such a solution exhibits, and a decision maker may have a certain preference with respect to that qualities. We chose the name \textit{uncertainty preferences} as an overarching title for such preferences, regarding the structure of the uncertainty set and the implied quality of robustly optimal solutions. Let us illustrate the core idea of our paper in  a simple  yet illustrative example.
	\begin{exmp}\label{exmp:ShortestPath}
		Consider the graph depicted in \cref{fig:exmp1}. The travel costs along the edges are uncertain, and only known to lie in a given interval. We can choose to influence uncertainty either on the edges that lead to $T_1$ or on edges that lead to $T_2$. We refer to the uncertainty regime where no influence on the uncertainty is taken as $\UU_0$, the one where uncertainty is affected on paths towards $T_i$ is referred to as $\UU_i$, where $i=1,2$.  Our aim is to choose a path from the source-vertex $S$ to one of the sink-vertices $T_1$ or $T_2$ such that the worst-case traveling time is minimized. The following table gives the travel costs on the edges under different uncertainty regimes.		
		\begin{figure}[ht]
			\begin{subfigure}{0.5\textwidth}
				\hspace{2cm}
				\begin{tikzpicture}
					[
					vertex/.style = {draw, inner sep = 2pt, circle},
					arc/.style = {very thick, ->},
					label/.style = {sloped, above, font=\small}
					]	 			
					\node (1) [vertex] at (0, 0) {$S$};
					\node (2) [vertex] at (0, 2) {$T_1$};
					\node (3) [vertex] at (2.5 , 0) {$T_2$};	 			
					\draw (1) edge [arc, bend right, right] node  {$E_2$} (2);
					\draw (1) edge [arc, bend left, left] node  {$E_1$} (2);
					\draw (1) edge [arc] node [label] {$E_3$} (3);
				\end{tikzpicture}
			\end{subfigure}
			\begin{subfigure}{0.5\textwidth}
				\textnormal{
					\begin{tabular}{c|c c c}
						Edge&$\UU_0$&$\UU_1$&$\UU_2$  \\
						\hline
						$E_1$&[4,5] &[3,4] &[2,5]\\
						\hline
						$E_2$&[4,5] &[2,4] &[2,5]\\
						\hline
						$E_3$&[4,5] &[3,5] &[1,4]\\
				\end{tabular}}
			\end{subfigure}
			\caption{}\label{fig:exmp1}	
			\label{fig:image2}
		\end{figure}	 	
			Unless we impose a cost on reducing uncertainty, $\UU_0$ is the inferior choice to make. Thus, one would either choose $\UU_1$ or $\UU_2$ as the uncertainty regime under which to operate and then choose only among those paths, that are optimal in the worst case which leaves us with only three choices described in below table.
			
			The core realization is that, depending on the uncertainty set, the robustly optimal solutions vary substantially with respect to characteristics other than worst-case performance. In fact, in this example, the three robustly optimal decisions do not vary at all with respect to worst-case performance. A decision maker who only seeks to optimize the worst case behavior of the solution would be indifferent between the three choices. However, each of these choices can be clearly ranked with respect to maximum-regret, best-case performance and predictability, i.e.\ the range of possible outcomes.
		
		\begin{figure}[h]
			\begin{tabular}{c||c| c c c}
				\textbf{Decision}& \textit{worst-case perf.}  & \textit{maximum-regret} & \textit{best-case perf.} & \textit{predictability} \\
				\hline
				traverse $E_1$ under $\UU_1$& 4 & 2 & 3 & \textbf{1}\\
				\hline
				traverse $E_2$ under $\UU_1$& 4 & \textbf{1} & 2 &  2\\
				\hline
				traverse $E_3$ under $\UU_2$& 4 & 3 & \textbf{1} &  3\\ 			
			\end{tabular}
		\end{figure} 		
			\end{exmp}	
	
	\cref{exmp:ShortestPath} shows that for different uncertainty regimes, robust solutions to the shortest path problem may exhibit additional features beyond robustness, which may be considered when evaluating such choices. Specifically, under different uncertainty regimes the respective robustly optimal solutions differ with respect to the \textit{predictability} of the outcome, the risk of experiencing \textit{regret} in hindsight, and the \textit{best-case performance}. We can also see that under some uncertainty regimes, a robust solution is not unique, and some of the robustly optimal choices may also exhibit said distinguishing features. We conceptualize the notion of uncertainty preference as the desire to operate under uncertainty regimes where there is a robustly optimal solution that performs well with respect to predictability, regret and/or best-case performance.  

	It is plausible that a decision maker faces this situation under various circumstances. Firstly, robustly optimal solutions may be similar across uncertainty regimes, yet still vary significantly with respect to other characteristics mentioned above. In this case a decision maker may wish to choose, among several nearly identical robustly optimal solutions, those that also incur additional advantages. Secondly, the decision on the uncertainty regime and on the solution to be implemented may be divided between two parties. This could take different forms: having a solution that is worst case optimal for the chosen uncertainty regime could be a requirement that is forced upon the decision maker, for example by company policy. In such a situation, a decision maker may still want to optimize other characteristics via their influence on the uncertainty regime and/or by picking a respective solution from the robustly optimal set. Further, in a setting similar to the classical leader-follower arrangement of bilevel optimization a decision maker who only has influence over the uncertainty regime may want to optimize their choice under the assumption that whoever is in control over the other decision variables will choose a robustly optimal solution.

		We will tackle this phenomenon under two major paradigms, namely that of \textit{multi-objective optimization} and that of \textit{bilevel optimization}. Again we want to use examples to illustrate our reasoning.
		\begin{exmp}\label{exmp:BilevelExample}
			A decision maker is tasked with an investment in public infrastructure such that in case of a natural disasters, important supply routes can be traversed confidently. The investment would reduce uncertainty in the travel time between two neuralgic points $A$ and $B$ that are connected through a network of roads, so that different paths exist between the two. There are not enough funds to fortify all roads. The decision maker knows that in case of an emergency the only acceptable travel route is the robustly optimal one. But at the same time they know that after the fact, their decision will be evaluated with the power of hindsight. In such a situation a decision maker may wish to take a decision on the investment in public infrastructure in such a way that the robustly optimal solution exhibits low risk of regret, such that it withstands the evaluation made in hindsight. Such a decision requires the anticipation of the robustly optimal solution present under the choice on the uncertainty regime that is to be made by the decision maker. Thus, we have a bilevel optimization problem.
		\end{exmp}
		\begin{exmp}\label{exmp:EpsilonConstraintExample}
			In the above example, even if the decision maker does not bear the risk of unfair evaluation, there may be several robustly optimal solutions under the robustly optimal uncertainty regime, and a decision maker may want to optimize a secondary characteristics over the optimal set of the primary decision problem, or an appropriately chosen set containing the latter. This is a classical problem in multi-objective optimization and can be met by employing so called $\eps$-constraints.
		\end{exmp}
	
	Following the ideas behind these examples, the goal of this paper is to provide models that capture uncertainty preference in a robust optimization problem, and to give examples where these models yield tractable optimization problems which can be handled with standard methods such as conic mixed-integer optimization.
	
	\subsection{Contribution and organization of this paper}
		
		We introduce the notion of uncertainty preference, which accounts for the fact that robustly optimal solutions obtained under different uncertainty regimes may have secondary, yet relevant characteristics (see Subsection~\ref{sec:Uncertainty preferences and main motivation of this study})  
		and formulate two mathematical models that capture the respective optimization problem: a multi-objective optimization model involving $\eps$-constraints and a bilevel optimization model (see Subsection~\ref{apx:EndogenousUncertainty}). 
		
		In case the nominal problem is a linear optimization problem, we characterize the set of robustly optimal solutions based on robust optimization duality theory. The resulting reformulation of the bilevel optimization problem is amenable to standard C-MILP solvers. In case the nominal problem is a MILP, we give a finite reformulation of the bilevel optimization problem with a potentially exponential number of constraints. These reformulations can be tackled using lazy constraints in conjunction with C-MILP solvers. These topics are discussed in Sections~\ref{sec:WorkingwithRxB} and \ref{sec:Finite reformulations for mixed-integer problems}.
		
		In Section~\ref{sec:ModellingUncertaintyPreferenceviaMM}, we describe the objective function terms that capture preferences for predictability, max-regret and best-case performance. Except for predictability, these formulations are in general not easy to handle. We highlight some special cases where we can achieve reformulations or approximations that can be solved with reasonable effort using C-MILP solvers. We apply our framework in Section~\ref{sec:RSPP} to the shortest path problem under endogenous uncertainty, which was considered before in \cite{lappas2018robust} and will serve as an extensive illustration of our framework. Numerical experiments are performed in Section~\ref{sec:NumericalExperiments} basically as a proof of concept. We demonstrate that the proposed reformulations are handled well by C-MILP solvers in interesting and relevant cases, that the results obtained are meaningful and hopefully can aid a decision process.

		We want to highlight that we decided, in agreement with the editor, to relegate some additional auxiliary material to an appendix, which can be found in the extended technical report on Arxiv.  A short discussion on the generalization of the theory discussed in Section~\ref{sec:RSPP} to other network flow problems is presented in Appendix A. Furthermore, to give an alternative account of the applicability of our framework, in Appendix B, where we address a classical example from \cite{bertsimas2004price}, the knapsack problem under endogenous $\Gamma$-uncertainty. For readers who are not familiar with standard robust optimization techniques we put the otherwise trivial proof of \cref{thm:Predictability} in Appendix C. Some open questions are also briefly discussed in Appendix D.

	\subsection{Notation and preliminaries}
	In the sequel, matrices are denoted by sans-serif capital letters (e.g.\ the $n\times n$-identity matrix will be denoted by $\Ib_n$ and $\Ob$ will be the zero matrix), vectors by boldface lower case letters (for example \ $\e_i$ will denote the $i$-th column of $\Ib_n$, $\oo$ will denote the zero vector, and $\e$ the vector of all ones) and scalars (real numbers) by lower case letters. The $i$-th entry of a vector $\x$ will be denoted by $x_i$. Matrices and vectors, and matrices and numbers with double indices are treated analogously. That is, for a matrix $\Xb$, the $i$-th column or row vector (usage will be pointed out in context) is denoted by $\x_i$, and the $j$-th entry in the $i$-th row will by given $x_{ij}$. If the original vector already has an index, that index will be pushed into the exponent for the lower case symbol that denotes an entry, e.g. the $i$-th entry of $\ggd_j$ is denoted by $\gd^j_i$. Sets will be denoted using calligraphic letters, e.g., cones will often be denoted by $\KK$ and uncertainty sets by $\UU$. The only exception is the set $\irg{1}{k} \coloneqq \lrbr{1,\dots,k}$, and the function $\MM(.,.)$, which is explained later in the text and which we chose to highlight due to its significance.
	
	Throughout the text, we will make use of conic linear optimization tools. A conic linear optimization problem is given by
	\begin{align}\label{eqn:conicprimal}
		p^*=\inf_{\x}\left\lbrace \langle\cc,\x\rangle \colon \x\in \KK,\
		\langle\aa_i,\x\rangle = b_i\, ,\mbox{ all } i \in\irg{1}{m}\right\rbrace \, ,
		\tag{P}
	\end{align}
	which is just a linear optimization problem with an extra constraint that restricts the decision variable $\x$ to lie in a closed, convex cone $\KK$. The dual problem is given by
	\begin{align}\label{eqn:conicdual}
		d^*=\sup_{\ggl} \left\{ \sum_{i=1}^{m}\gl_i b_i \colon \ggl\in \R^m,\
		\cc - \sum_{i=1}^{m}\gl_i\aa_i \in \KK^* \right\}\, ,
		\tag{D}
	\end{align}
	where $\KK^*\coloneqq \lrbr{\x\in\R^n\colon \y\T\x\geq 0 \ \mbox{ all } \y\in\KK }$ is the dual cone of $\KK$.
	Slater's condition says that a relative interior feasible point of (P) guarantees $p^*=d^*$ and existence of an optimal solution to (D) if it is feasible, and likewise with (D) and (P) in reversed roles. The gap between the two also vanishes if $\KK$ is polyhedral and one of the problems is feasible.
	
	A central object this paper is concerned with is the \textit{endogenously uncertain (mixed-integer) linear optimization problem},
	\begin{align}\label{eqn:EndoUncer}
		\inf_{\x} \lrbr{ \cc(\u)\T\x \colon \ \x \in \XX_0, \ \aa_i(\u)\T\x \leq b_i(\u) \, ,\mbox{ all } i \in \irg{1}{m} }\  \mbox{ where } \u \in \UU^s,\, \mbox{and } s\in \irg{1}{S}\,  .			 \end{align}
	Here the problem data are affine functions of $\u$, i.e.
	$\cc(\u)\coloneqq\cc_0 + \Cb\u,\ \aa_i(\u) \coloneqq \aa^0_i + \Ab_i\u$ and $b_i(\u)\coloneqq b^0_i + \b_i\T\u\, ,\mbox{all }  i \in \irg{1}{m}$ and $\XX_0 \in \lrbr{\R^n_+,\ \R^{n_1}_+\times \lrbr{0,1}^{n_2}}$. Problem~(\ref{eqn:EndoUncer}) describes a family of instances parametrized by $\u$ and $s$. The former is called the \textit{uncertainty parameter}, which is not under the control of the decision maker who considers only those realizations of $\u$ which belong to  certain \textit{uncertainty sets} $\UU^s\subseteq \R^{q}$. Here we assume that these sets consist of all realizations of $\u$ the decision maker need to take into account, according to their responsibility.
	
	The uncertainty sets $\UU^s$ depend on the second parameter
	$s \in \irg{1}{S}$ which in fact is controlled by the decision maker, so that they can choose the \textit{uncertainty regime} under which they would like to operate. The uncertainty is thus endogenous. Throughout the paper we will model uncertainty sets as compact, convex conic intersections, i.e.  $\UU^s \coloneqq \lrbr{\u \colon (1,\u)\T \in \KK^s}$ where $\KK^s$ are appropriate closed, convex, pointed cones.
	Note that these assumptions, together with boundedness of $\UU^s$ entail the following implication:
	\begin{align}\label{eqn:0coneimplies0}
		(0,\x\T)\T \in \KK^s \implies \x = \oo \, .
	\end{align}
	\begin{rem}\label{rem:RemarkonModellingChoice}
Our modeling choice on  $\UU^s $, $\KK^s$ and $s\in\irg{1}{S}$ was made with the desire in mind to have a readable notation that can, with little effort, capture very general structure. We are aware that these are neither the most natural nor the most practical modeling choices for those objects, if we think of more specific instances of robust optimization problems under endogenous uncertainty. Consider for example uncertainty sets $\hat{\UU}(\hat{\r}) \coloneqq \lrbr{\u \in \R^q_+ \colon u_i \leq 1 + 0.5\hat{r}_i}$ with $\hat{\r}\in \hat{\XX}\coloneqq\lrbr{\hat{\r}\in \lrbr{0,1}^{\hat{S}}\colon \Ab\hat{\r}\leq \b}$ and assume the latter set is bounded. Then one can set $S = |\hat{\XX}|$, construct an isomorphism $f \colon \hat{\XX} \mapsto \irg{1}{S}$ and set $\KK^s\coloneqq \mathrm{cone}\left(\lrbr{1}\times \hat{\UU}\left(f^{-1}(s)\right)\right)$ in order to obtain $\UU^s = \hat{\UU}\left(f^{-1}(s)\right)$ so that we have a reformulation that complies to our assumption but is arguably less practical. Firstly, the number of variables has increased to a size perhaps exponential in $\hat{n}_B$ and the size of $\Ab$. Secondly the implicit constraint $\e\T\r = 1$ present in  our formulation can be much more challenging for MILP solvers than the inequalities in the description of $\hat{\XX}$. For practical purposes it is therefore advisable to operate in the original space of variables. We chose our general formulation merely to be able to speak of general structure in simple terms.		 
	\end{rem}	
	
	The construction of uncertainty sets is a well studied field and might be guided by different motivations. The most basic approach is to construct them in a way that guarantees that the robustified constraint is not violated with a certain probability (see \cite[Section 3]{gorissen2015practical}  and references therein). However, other authors have discussed construction schemes based on other principles such as reformulating optimization problems involving certain risk measure (see e.g.\ \cite{bertsimas2009constructing}). In the endogenously robust setting, \cite{lappas2018robust} give a detailed account on how uncertainty sets can be influenced by the decision maker and how to model these situations. We want to point out that it may not make sense to talk about predictability, max-regret, or best case performance over an uncertainty set that is constructed based on reformulating an optimization problem involving risk measures. These characteristics only make sense if the uncertainty set can be interpreted as a subset of the possible realizations of the uncertain process. We further comment on this issue in the final section of the appendix.

	We will employ several concepts associated with (\ref{eqn:EndoUncer}), namely: \textit{row/column wise uncertainty}, the \textit{nominal problem}, \textit{robust/optimistic/deterministic counterparts} as well as C-MILP reformulations of those. We refer to \cite{beck2009duality,ben2009robust,lappas2018robust,nohadani2018optimization,soyster2013unifying,soyster2016integration} for some background on these concepts.
	 	Frequently, for example in case of row- and column-wise uncertainty, the uncertainty set will be a Cartesian product of sets $\UU^s = \UU^s_1\times\dots\times\UU^s_k$. The factors of such products are likewise  modeled as conic intersections with cones $\KK_i^s,\ i \in \irg{1}{k}$.
	
	\begin{rem}\label{rem:RemarkofEntagledVariables}
		The choices  on $\x$ and $s$ a decision maker takes may be entangled in a way  that a decision on $\x$ limits the options on $s$  (see \cite{nohadani2018optimization} for examples). Such a situation can always be modeled via constraints of the form
	\begin{align*}
		\Hb\x+\Gb\y \leq \Fb\r_s+\f,
	\end{align*}
	where $\y$ is an auxiliary (real or binary) variable; $\Hb,\Gb,\Fb,\f$ are appropriate matrices/vectors and the entries of $\r_s$ are defined by the indicators (binary variables)
	\begin{align}\label{eqn:DefinitionOfr}
		r^s_i =
		\begin{cases}
			1 & \mbox{ if } s = i\, ,\\
			0 & \mbox{ else.}
		\end{cases}
	\end{align}
	However,  for the sake of notational simplicity, we will omit these constraints from the models discussed below.  Also, in some applications it will be practical to push the dependence on $s$ into the functions $\cc(\u),\ \aa_i(\u)$ and $b_i(\u)$ while keeping the uncertainty sets constant (see Section~\ref{sec:RSPP}).
	\end{rem}

	\section{The model}
		\subsection{Endogenous uncertainty}\label{apx:EndogenousUncertainty}
		
		In the case of \textit{endogenous uncertainty} a decision maker has some level of influence over the shape of the uncertainty set. So they can choose the \textit{uncertainty regime} under which they would like to operate.  For linear and mixed-integer linear optimization problems under endogenous uncertainty, the mathematical model is thus given by $\inf\limits_{s \in [1:S]} \left (\varphi_s+\gc_s\right )$, where
		\begin{align}\label{eqn:UAModelPrimalWorst}\tag{RC-s}
			\begin{split}
				\varphi_s \coloneqq \inf_{\x} \quad  \sup_{\u_0\in \UU^s_{0}} & \lrbr{ \cc(\u_0)\T\x} \\
				\mathrm{s.t.:}\ \aa_i(\u_i)\T\x &\leq b_i(\u_i) \,  \mbox{ for all } \u_i \in \UU^s_i \, ,\mbox{ all }  i \in \irg{1}{m}\\
				\quad \x &\in \XX_0.
			\end{split}		
		\end{align}
		Under uncertainty regime  $s$, the term $\sup_{\u_0\in \UU^s_{0}} \lrbr{\cc(\u_0)\T\x}$ captures the worst-case performance of a solution $\x\in\XX_0$, and $\varphi_s$ is the optimal value of the
		robust counterpart to~(\ref{eqn:EndoUncer}).  The cost of activating uncertainty regime $s$ is given by $\gc_s$, which is thus a function of $s$. C-MILP reformulations of $\inf\limits_{s \in [1:S]} \varphi_s+\gc_s $ have been studied in \cite{lappas2018robust,nohadani2018optimization}. We define the set of optimal solutions to this problem to be
		\begin{align*}
			\RR^s_{opt} \coloneqq \lrbr{\x \in \R^n_+ \colon \x \mbox{ is an optimal solution to (\ref{eqn:UAModelPrimalWorst})}}	
		\end{align*}
		We will discuss in Subsection~\ref{sec:WorkingwithRxB} how to work with the constraint $\x\in\RR^s_{opt}$. Following the ideas conveyed in \cref{exmp:BilevelExample} and \cref{exmp:EpsilonConstraintExample}, we will now introduce two models in the focus of this text.
		
		\subsubsection{A multi-objective optimization approach via $\eps$-constraints}
		One phenomenon arises quite frequently in the endogenously uncertain settin: the different choices on the uncertainty regime and, subsequently, on the decision variables may perform similar with respect the worst case criterion, but may vary greatly with respect to other characteristics. Multi-objective optimization is the standard paradigm under which we consider problems where more than one characteristics of a solution is to be optimized. In this setting, the objective function is a vector of the respective characteristics of a solution.
				Ordering cones are used to establish a type of dominance between different solutions.
		However, we wish to focus on the robust optimization perspective. Hence, we want to look for (nearly) robustly optimal solutions that optimize secondary characteristics. To this end we use the concept of $\eps$-constraints which is well known in the multi-objective optimization literature (see e.g. \cite{emmerich2018tutorial}) . The respective model is given by:
		\begin{align*}
			\begin{split}
				\inf_{\x , s } \quad \gc_s + \MM(\x,s)   \\
				\mathrm{s.t.:}\ \aa_i(\u_i)\T\x &\leq b_i(\u_i) \,  \mbox{ for all } \u_i \in \UU^s_i \, ,\mbox{ all }  i \in \irg{1}{m}\\
				\sup_{\u_0\in \UU^s_{0}} \lrbr{ \cc(\u_0)\T\x}& \leq \varphi^{*} + \eps, \
				\x \in \XX_0.
			\end{split}
		\end{align*}
		Here, $\varphi^*$ is the optimal value of the endogenously robust counterpart, gross of the activation cost of the robustly optimal uncertainty regime, and $\eps\geq 0$ is a degree two which we allow our solution to deviate from that optimal value. The function $\MM(\x,s)$ gives a measure of the performance of a solution $(\x,s)$ with respect to predictability, max-regret and best-case performance. This function influences the decision in two ways: firstly, which uncertainty regime is implemented and secondly, which robustly optimal solution is chosen in case the selected uncertainty regime allows for multiple robustly optimal solutions. We will explore different choices for $\MM(\x,s)$ and study their reformulations in Section~\ref{sec:ModellingUncertaintyPreferenceviaMM}. But first we will introduce an alternative model that is applicable to bilevel optimization setting.
		
		\subsubsection{A bilevel optimization approach}\label{sec:A bilevel optimization approach}
		Following the idea illustrated in \cref{exmp:BilevelExample}, we consider a leader-follower situation where the leader has control over the uncertainty regime and wants to optimize their preferred attribute under the assumption that the follower will implement the robustly optimal solution under the chosen uncertainty regime. Thus, in the leader's optimization problem, we need to guarantee that we only consider robustly optimal solution under the uncertainty regime of choice for the decision variable $\x$ controlled by the follower. Thus we want to solve
		\begin{align}
			\begin{split}\label{eqn:UPMModel}
				\inf_{\x,s} & \lrbr{\gc_s +  \ga\sup_{\u_0\in \UU^s_{0}}
				\cc(\u_0)\T\x
				+ \MM(\x,s) \colon \x \in \RR^s_{opt},\ s \in \irg{1}{S} } \, .		
			\end{split}\tag{UPP}
		\end{align}
		Note  that (\ref{eqn:UPMModel}) is feasible if and only if there is at least one $s\in\irg{1}{S}$
		such that~(\ref{eqn:UAModelPrimalWorst}) is feasible and attains its optimal value.	
		Here $\ga$ is a weight given to the robustified  objective value. We would like to point out that our model assumes a cooperative relationship between leader and  follower. In fact, for any fixed uncertainty regime the set of robustly optimal solutions may not be a singleton, and our model will select among these solutions the one optimizing the objective. However, an adversarial follower might not choose this solution once the leader implemented their preferred uncertainty regime. \cref{exmp:BilevelExample} depicts a situation where the assumption of a cooperative follower is justified. Modelling an adversarial follower is a more challenging task, which we will comment on in in the appendix.
		
	\subsection{The robustly optimal set $\RR^s_{opt}$}\label{sec:WorkingwithRxB}
	The feasible set of the decision vector  $\x$ is  given by $\RR^s_{opt}$, i.e. the set of solutions that are robustly optimal under the uncertainty regime indexed by  $s$. Hence, we have to solve a bilevel optimization problem. This class of problems is known to be challenging, in fact it is strongly NP-hard, as shown in \cite{hansen1992new}. Even if the optimal set of a lower level problem can be characterized, its parametrization via the top-level decision may lead to problem formulations that are not easily handled due to arising bilinear terms.
	
	Luckily, the top-level decision on $s$ is discrete, so that these bilinearities can be  exactly linearized by using McCormick envelopes~\cite{McCormick1976}. Thus, the characterization of the robustly optimal set for a fixed choice of $s$ is the main issue, which we will address in the sequel here. For cases where $\XX_0 = \R^n_+$ we will use robust duality theory in order to characterize $\RR^s_{opt}$. In case of $\XX_0 = \R^{n_1}_+\times\lrbr{0,1}^{n_2}$, this theory does unfortunately not apply, so we provide alternative ways for these cases.
	
 	Let us first deal with the case when $\XX_0 = \R^n_+$. In recent literature many aspects of duality in robust optimization have been studied (see \cite{beck2009duality,soyster2013unifying}). While the deterministic counterpart of a robust optimization problem, as a convex problem, enjoys rich duality  theory, a somewhat surprising observation first made in \cite{beck2009duality} is that, under mild assumptions, the dual of the robust counterpart of an uncertain linear optimization problem can be identified as the optimistic counterpart of the dual of that optimization problem (where the dualization is with respect to $\x$). Hence the phrase \textit{"primal worst equals dual best"} was coined in \cite{beck2009duality}. In case of (\ref{eqn:UAModelPrimalWorst}) the corresponding "dual best" problem is given by
	\begin{align}\label{eqn:DualBest}
		\begin{split}	
			\sup_{\y,\u_i} \lrbr{ \sum_{i=1}^{m}b_i(\u_i)y_i \colon
			\mathrm{s.t.:} \sum_{i=1}^{m}\aa_i(\u_i)y_i \leq \cc(\u_0)\, ,\
			\y\in \R^m_+,\  \u_i\in \UU^s_i \, ,\mbox{ all }  i \in \irg{0}{m}}\, .
		\end{split}	
	\end{align}
	We have the following result \cite[Theorem 3.1]{beck2009duality}:
	\begin{thm}\label{thm:PrimalworstEqualsDualbest}
		The problem~(\ref{eqn:DualBest}) and the Lagrangian dual of (\ref{eqn:UAModelPrimalWorst}) are equivalent.
	\end{thm}
	Note that (\ref{eqn:UAModelPrimalWorst}) is convex in $\x$ as the inner suprema are convex functions in $\x$. Its deterministic counterpart  is obtained by merely finding explicit forms of those functions (a standard procedure in robust optimization). Thus, the dual of (\ref{eqn:UAModelPrimalWorst}) is in fact given by the dual of deterministic counterpart, and any condition for this deterministic counterpart to exhibit strong duality also implies strong duality between (\ref{eqn:UAModelPrimalWorst}) and (\ref{eqn:DualBest}). 	
	\begin{cor}\label{cor:CharacterizingRxB}
		Assume that strong duality holds between (\ref{eqn:UAModelPrimalWorst}) and (\ref{eqn:DualBest}) and that both problems attain their optimal value. Then the optimal primal-dual pairs $(\x;\bar{\u},\y)\in \R^n_+\times \UU^s\times \R^m_+$ of (\ref{eqn:UAModelPrimalWorst}) are characterized by
		\begin{align*}
			\sup_{\u_0\in \UU^s_{0}}\lrbr{\cc(\u_0)\T\x} &\leq \sum_{i=1}^{m}b_i(\bar{\u}_i)y_i\, , \hspace{6.2cm} &\mbox{(Zero Duality Gap)} \\
			\aa_i(\u_i)\T\x &\leq b_i(\u_i)\,  \mbox{ for all } \u_i \in \UU^s_i \, ,\mbox{ all }  i \in \irg{1}{m}\, ,  
			&\mbox{(Primal Feasibility)} \\
			\sum_{i=1}^{m}\aa_i(\bar{\u}_i)y_i &\leq \cc(\bar{\u}_0) \, . 
			&\mbox{(Dual Feasibility)}
		\end{align*}		
		These relations therefore characterize $\RR^s_{opt}$. \end{cor}
	\begin{proof}
		The proof is immediate from \cref{thm:PrimalworstEqualsDualbest}.
	\end{proof}
	The inequalities in the above corollary may seem difficult, but we have to bear in mind that all of these have convex reformulations.
	Minding the discussion succeeding \cref{thm:PrimalworstEqualsDualbest} we can see that the assumptions in \cref{cor:CharacterizingRxB} can be guaranteed under well known conditions for full strong duality from convex duality theory.  
	
	Let us highlight at this point that not all optimistic counterparts are duals of robust counterparts as defined here. In fact, the structural aspect of (\ref{eqn:DualBest}) which allows for duality (and, thus, a convex reformulation) is that the uncertainty is column-wise (see \cite[Theorem~1]{soyster2013unifying}). This fact will allow us to reformulate certain models involving max-regret in Subsection~\ref{sec:Regret} below.

	\subsection{Finite reformulations for mixed-integer problems}\label{sec:Finite reformulations for mixed-integer problems}
	
	We now consider the case where $\XX_0 = \R^{n_1}_+\times\lrbr{0,1}^{n_2}$ so that (\ref{eqn:UAModelPrimalWorst}) is a mixed-integer robust optimization problem. In this case, robust duality theory does not apply. However, it is still possible to give a finite reformulation of (\ref{eqn:UPMModel}) with possibly exponentially many constraints which can be tackled via cutting plane algorithms, especially by using lazy constraints in a branch and bound algorithm. The following theorem is easily checked:
	
	\begin{thm}\label{thm:BinrayRef1}
		Assume $\XX_0 = \R^{n_1}\times\lrbr{0,1}^{n_2}$, then $(\x,s)$ is feasible for (\ref{eqn:UPMModel}) if and only if  the following constraints hold:
		\begin{align}\label{eqn:BinrayRef}
			\begin{split}
				\sup_{\u_0\in \UU^s_{0}}\lrbr{\cc(\u_0)\T\x}&\leq \varphi_j +(1-r^s_j)M_j \, ,\mbox{ all } j\in\irg{1}{S},\\
				\aa_i(\u_i)\T\x&\leq b_i(\u_i) \, \mbox{ for all } \u_i \in \UU^s_i \, ,\mbox{ all } i \in \irg{1}{m},\\
				\x&\in \XX_0,\ s \in \irg{1}{S},
			\end{split}
		\end{align}	
		where $M_j$ are upper bounds for $\varphi_{\max} - \varphi_j\, ,\mbox{ all }  j \in \irg{1}{S}$ with
		 		$\varphi_{\max}\coloneqq \sup\limits_{s\in [1:S]} \varphi_s$, while 		 
		$\varphi_j$ is the optimal value of (\ref{eqn:UAModelPrimalWorst}) under uncertainty regime $s=j$ and $\r_j$ is defined as in~(\ref{eqn:DefinitionOfr}).
	\end{thm}
	\begin{proof}
		Let $(\x,s)$ be feasible for (\ref{eqn:UPMModel}), then $s \in \irg{1}{S}$ and $\x\in \RR^s_{opt}$ so that it fulfills the last two constraints of (\ref{eqn:BinrayRef}). Further, by this fact, we have $\sup_{\u_0\in \UU^s_{0}}\lrbr{\cc(\u_0)\T\x} = \varphi_s \leq\varphi_{\max}\leq \varphi_k + M_k$
		for all $k \in \irg{1}{S}\setminus \lrbr{s}$. Thus $(\x,s)$ is feasible for (\ref{eqn:BinrayRef}). For the converse, assume $(\x,s)$ is feasible for (\ref{eqn:BinrayRef}) so that all but the $s$-th constraint are redundant and by definition and by feasibility we have $\varphi_s \leq \sup_{\u_0\in \UU^s_{0}}\lrbr{\cc(\u_0)\T\x}\leq \varphi_s$. So that indeed $\x\in \RR^s_{opt}$, and $(\x,s)$ is feasible for (\ref{eqn:UPMModel}).
	\end{proof}
	Finding appropriate $M_j$ for all $j \in \irg{1}{S}$ can be challenging. At least $\varphi_j$ can simply be computed by solving a robust optimization problem, that has to be computed anyway whenever one of the constraints from the first group in \cref{thm:BinrayRef1} is added. However, $\varphi_{\max}$ is the more challenging entity. One way to derive an upper bound would be to solve $\sup_{s\in [1:S]} \bar{\varphi}_s$ where $\bar{\varphi}_s$ is obtained by changing the outer infimum in $\varphi_s$ to a supremum. The resulting problem is a mixed integer bilinear problem that can be exactly reformulated as a mixed integer linear problem, which in turn can be solved using the respective standard solvers. In Appendix B we discuss an example where a cheaper bound is derived easily.
	However, if (\ref{eqn:UAModelPrimalWorst}) is such that the uncertainty in the constraints does not depend on $s$, we can achieve a general reformulation that does not require a big-M. Note that this covers the case  where there is no uncertainty in the constraints.
	\begin{thm}\label{thm:BinrayRef2}
		Assume $\XX_0 = \R^{n_1}_+\times\lrbr{0,1}^{n_2}$ and $\UU^s_i = \UU_i$ for $ i\in\irg{1}{m}$, i.e. only the uncertainty in the objective of (\ref{eqn:EndoUncer}) is affected by $s$. For all $j\in\irg{1}{S}$, select any $\z_j\in \RR^j_{opt}$. Then $(\x,s)\in \XX_0\times \irg 1S$ if feasible for (\ref{eqn:UPMModel}) if and only if the following constraints hold:
		\begin{align}\label{eqn:BinrayRef2}
			\begin{split}
				\sup_{\u_0\in \UU^s_{0}}\lrbr{\cc(\u_0)\T\x}&\leq \sup_{\u_0\in \UU^s_{0}}\lrbr{\cc(\u_0)\T\z_j} \, ,\mbox{ all }j\in\irg{1}{S}\, ,\\
				\aa_i(\u_i)\T\x&\leq b_i(\u_i) \, \mbox{ for all } \u_i \in \UU_i \, ,\mbox{ all } i \in \irg{1}{m}\, .\\
			\end{split}
		\end{align}
	\end{thm}
	\begin{proof}
		Assume $(\x,s)$ is feasible for (\ref{eqn:UPMModel}), then $\x\in \RR^s_{opt}$ so that it fulfills the last constraint of (\ref{eqn:BinrayRef2}). Also by this fact we have that $\varphi_s = \sup_{\u_0\in \UU^s_{0}}\lrbr{\cc(\u_0)\T\x}= \sup_{\u_0\in \UU^s_{0}}\lrbr{\cc(\u_0)\T\z_s}\leq \sup_{\u_0\in \UU^s_{0}}\lrbr{\cc(\u_0)\T\z_j}$ for all $j\in\irg{1}{S}\setminus\lrbr{s}$, since all $\z_j$ are feasible for (\ref{eqn:UAModelPrimalWorst}). Thus $(\x,s)$ is feasible for (\ref{eqn:BinrayRef2}). Conversely, assume $(\x,s)$ are feasible for (\ref{eqn:BinrayRef2}). We have $\sup_{\u_0\in\UU_{0}^s}\lrbr{\cc(\u_0)\T\z_s} \leq \sup_{\u_0\in \UU_{0}^s}\lrbr{\cc(\u_0)\T\x}\leq \sup_{\u_0\in \UU_{0}^s}\lrbr{\cc(\u_0)\T\z_s}$ where the first inequality holds since $\x$ is feasible for $\varphi_s$ and the second holds by feasibility for (\ref{eqn:BinrayRef2}). So that indeed $\x\in \RR^s_{opt}$, and $(\x,s)$ is feasible for (\ref{eqn:UPMModel}).
	\end{proof}

	\section{Modeling uncertainty preference}\label{sec:ModellingUncertaintyPreferenceviaMM}
	In this section we will discuss different choices for $\MM(\x,s)$ and various reformulations of the resulting optimization problems. The goal is to have exact reformulations or at least good relaxations that are amendable to various, general off-the-shelf solution techniques such as mixed-integer linear and conic optimization. For some of the models, row-wise uncertainty can be assumed without loss of generality, while in others that is not generally possible. We will point specifically, when this assumption can be made safely and when this is not the case.  	
	
	\subsection{Regret}\label{sec:Regret}
	
	We can think of situations where the ex-ante choice a decision maker would have made in hindsight heavily depends on the outcome of the uncertainty process, while in other situations, a worst-case robust choice is viable in cases other than the worst case. In the former situation, decision makers run the risk of regretting their decision and may prefer an uncertainty regime with low potential regret. The concept of regret is well studied in the robust optimization literature (see \cite{kouvelis1997robust}). The min-max-regret approach has garnered much attention as an alternative concept to the min-max approach in robust optimization. However, our aim is not to optimize the maximum regret in an uncertain linear problem. Rather we seek a choice of the uncertainty regime such that the robustly optimal solution performs well with respect to its maximum regret.
	
	Since a decision maker prefers uncertainty regimes with small possible regret, we now consider a model that introduces a cost for running the risk of high regret. For fixed $s\in\irg{1}{S}$ and $\x\in \RR^s_{opt}$ the worst-case regret associated with this solution is given by
	\begin{align*}
		\rho(\x,s) \coloneqq \sup_{\u\in \UU^s}\left [ \cc(\u)\T\x -\inf_{\y\in \R^n_+} \ \lrbr{\cc(\u)\T\y \colon \aa_i(\u)\T\y\leq b_i(\u)\, ,\mbox{ all }  i \in \irg{1}{m}}\right ]
	\end{align*}
	that is, the worst-case difference between the performance of $\x$ and the performance we could have achieved in hindsight. The regret function does not have an equivalent row-wise formulation based on the projection method, which is valid for the robust counterpart.
	
	The model that accounts for the worst-case regret of the outcome of the uncertainty under a chosen uncertainty regime can be formulated as:
	\begin{align}\label{eqn:REG-UPM}\tag{REG-UPP}
		\begin{split}
			&\inf_{\x,s,\rho}  \left [ \gc_s + \ga\sup_{\u_0\in \UU^s_{0}}  \lrbr{\cc(\u_0)\T\x} + \beta(\rho)\right ]\\
			\mathrm{s.t.:}\ \rho \geq  &\sup_{\u\in \UU^s}\lrbr{ \cc(\u)\T\x -\inf_{\y\in \R^n_+} \ \lrbr{\cc(\u)\T\y \colon \aa_i(\u)\T\y\leq b_i(\u)\, , \, \mbox{ all } i \in \irg{1}{m}}}\, ,\\
			\x &\in \RR^s_{opt},\ s \in \irg{1}{S}\, .
		\end{split}
	\end{align}
	Here, $\beta(\cdot)$ is a convex, non-decreasing function which introduces a cost on the regret associated with a chosen solution and $\rho$ is an auxiliary epigraphical variable. Easy viable choices for $\beta(\cdot)$ are the identity and simple linear scalings, but more complicated choices are possible without leaving the realm of convex optimization. The last constraint, despite being convex, presents a major complication in general, although in special cases there are manageable reformulations and approximations. In the remainder of the section we discuss problems with appropriate special structure. In Subsection~\ref{sec:SPIntervallUncertainty} we also use a reformulation from literature that is specific to the shortest path problem.
	
	\subsubsection{Exact reformulation in case of column-wise uncertainty}
	Note that the innermost supremum-problem in (\ref{eqn:REG-UPM}) is the optimistic counterpart of the uncertain linear problem (plus an additional robust objective term), which in general can be a nonconvex problem but in special cases is in fact convex, by robust duality theory. Obviously, this is a useful circumstance for reformulating  (\ref{eqn:REG-UPM}), but caution must be taken since not every optimistic counterpart is the dual of a robust optimization problem (see e.g \cite[Example 2.1, Remark 3.1]{beck2009duality}) and therefore convex. In fact, it was proved in \cite[Theorem 1]{soyster2013unifying} that an optimistic counterpart of an uncertain linear optimization problem with column-wise uncertainty and non-negative decision vector has a convex feasible set, that is, regarding the decision variable $\x$. Further, it was proved that such a problem is the dual of a robust counterpart of an uncertain linear optimization problem with row-wise uncertainty. Thus, in order for the minimization in the definition of $\rho(\x,s)$ to be the "dual best" of a robust linear optimization problem we need the following assumption on the underlying uncertain problem.
	\begin{asm}\label{asm:BoxUn}
		The underlying uncertain linear optimization problem has
		\begin{itemize}
			\item[a)] no uncertainty in the objective, i.e. $\cc(\u) = \cc_0$,
			\item[b)] column-wise uncertainty,
			\item[c)] the domain of $\x$ given by $\XX_0 = \R^n_+$,
		\end{itemize}
		and can thus be written as
		\begin{align}\label{eqn:ColumnWiseRobust}
			\inf_{\x}\lrbr{\cc_0\T\x\colon \x \in \R^n_+,\  \sum_{i=1}^{n} \aa_{i}(\u_i)x_i  \leq \b(\u_{n+1}) } \mbox{ where }\ \u_i \in \UU^s_i\, ,\mbox{ all }  i \in \irg{1}{n+1}.
		\end{align}
		where $\b(\u_{n+1}) \coloneqq \b_0+\Bb\u_{n+1}$.
		For the above formulation we assume that
		\begin{itemize}
			\item[d1)] either there is an $\x\in\R^n_+$ such that there exist $\u_i\in\mathrm{int}\left(\UU^s_i\right)\, ,\mbox{ all } i\in\irg{1}{n+1}$ so that the inequalities are fulfilled, where the $\mathrm{int}$-operator can dropped if the set is polyhedral, or
						\item[d2)] problem (\ref{eqn:ColumnsWiseDual}), i.e. the dual of (\ref{eqn:ColumnWiseRobust})  in the robust duality sense, has a Slater point.
		\end{itemize}
	\end{asm}
	The following lemma provides a useful sufficient condition for \cref{asm:BoxUn} \textit{d1)}:
	\begin{lem}\label{lem:SufficientCondforAss1}
		Consider the following condition on (\ref{eqn:ColumnWiseRobust}):
		\begin{itemize}
			\item[d')] There is an $\x\in\R^n_+$ such that there are $\u_i\in\UU^s_i\, ,\mbox{ all } i\in\irg{1}{n+1}$, so that the inequalities are fulfilled strictly.
		\end{itemize}
		We have that d') implies \cref{asm:BoxUn} \textit{d1)}.
	\end{lem}
	\begin{proof}
		Take any $\u^0_i \in \mathrm{int}(\UU^s_i)\, ,\mbox{ all }  i \in \irg{1}{n+1}$, then for $\bar{\u}_i \coloneqq\u^0_i-\u_i$ we have $\u_i+\gl\bar{\u}_i \in \mathrm{int}(\UU^s_i)$ for any $\gl \in (0,1]$ by \cite[Theorem 6.1]{rockafellar2015convex}. Define $\boldsymbol{\eps} \coloneqq \b(\u_{n+1})-\sum_{i=1}^{n} x_i\aa_i(\u_i)$. Since by assumption $\boldsymbol{\eps}\in \mathrm{int}(\R^m_+)$, we can choose $\gl\in (0,1]$ such that $\gl\left(\sum_{i=1}^{n}x_i\Ab_i\bar{\u}_i-\Bb\bar{\u}_{n+1}\right)\leq \boldsymbol{\eps}$. We then have
		\begin{align*}
			&\sum_{i=1}^{n} x_i\aa_i(\u_{i}+\gl\bar{\u})-\Bb\left(\u_{n+1}+\gl\bar{\u}_{n+1}\right)\\
						=& \sum_{i=1}^{n} x_i\left(\aa^0_{i} + \Ab_{i}\u_{i} \right) -\Bb\u_{n+1} +  \gl\left(\sum_{i=1}^{n}x_i\Ab_i\bar{\u} -\Bb\bar{\u}_{n+1}\right)\\
			\leq& \sum_{i=1}^{n} x_i\aa_i(\u_i)-\Bb\u_{n+1} +\boldsymbol{\eps} = \b_0
		\end{align*}		
		and $ \u_{i}+\gl\bar{\u}_i \in \mathrm{int}(\UU^s_i)\, ,\mbox{ all }  i\in \irg{1}{n+1}$.
	\end{proof}
	We can generally assume that $\lrbr{\oo}\in \mathrm{int}(\UU^s)$ since we can always shift the set and adapt the constant terms in $\aa(\u),\cc(\u)$ and $\b(\u)$ accordingly. Thus condition~d') holds if the feasible set of the nominal problem has nonempty interior, which often is established more easily. The more restrictive condition is $b)$, that is column-wise uncertainty. However, this case features prominently in literature and discussions on its relevance can be found in \cite{soyster2013unifying,soyster2016integration}.
	
	Note that the optimistic counterpart of (\ref{eqn:ColumnWiseRobust}) given by
	\begin{align}
		\begin{split}\label{eqn:ColumnWiseOptimist}
			\inf_{\x,\u} \lrbr{ \cc_0\T\x
			\sum_{i=1}^{n} x_i\aa_i(\u_i)  \leq \b(\u_{n+1}),\
			\x \in \R^n_+,\ \u_i\in\UU^s_i\, ,\mbox{ all } i \in \irg{1}{n+1} }\, ,
		\end{split}
	\end{align}
	is precisely of the form (\ref{eqn:DualBest}) so that by \cite[Theorem 3.1]{beck2009duality} we can close the duality gap between (\ref{eqn:ColumnWiseOptimist}) and
	\begin{align}\label{eqn:ColumnsWiseDual}
		\begin{split}\sup_{\z} \lrbr{ \inf_ {\u_{n+1}\in \UU^s_{n+1}} \lrbr{\b(\u_{n+1})\T\z} \colon
			c^0_i \geq \aa_i(\u_i)\T\z \, \mbox{ for all } \u_i \in \UU^s_i\, ,\mbox{ all }  i \in \irg{1}{n}\, ,\
			-\z \in \R^m_+} \, ,
		\end{split}
	\end{align}
	if regularity conditions are met (e.g. Slater's condition or polyhedrality).
	Mindful of this fact, we see that \cref{asm:BoxUn} suffices to prove the following theorem.
	\begin{thm}\label{thm:WBDCwiseRef}
		Assume that \cref{asm:BoxUn} holds. Then (\ref{eqn:REG-UPM}) is equivalent to the following optimization problem
		\begin{align}
			\begin{split}
				&\inf_{\x,\z,s,\rho} \left [\gc_s  + \ga \cc_0\T\x + \beta(\rho)\right] \\
				\mathrm{s.t.:}\ &\rho \geq \cc_0\T\x - \b\T\z + \tau\, ,\ -\z \in\R^n_+\, ,\ \x\in \RR^s_{opt}\, ,\ s \in \irg{1}{S}\,, \\
				&\left(\tau,\z\T\Bb\right)\T \in\left(\KK^s_{n+1}\right)^*\, ,\
				\left(c^0_i-(\aa^0_i)\T\z,\ -\z\T\Ab_i\right)\T \in \left(\KK^s_i\right)^*\, ,\mbox{ all }  i \in \irg{1}{n}\, .\\
			\end{split}
		\end{align}
	\end{thm}
		\begin{proof}
		Recall the definition of the max-regret function:
		\begin{align*}
			\rho(\x,s)=& \sup_{\u\in \UU^s}\lrbr{ \cc(\u)\T\x -\inf_{\y\in \R^n_+} \ \lrbr{\cc(\u)\T\y \colon \aa_i(\u)\T\y\leq b_i(\u)\, ,\mbox{ all } i \in \irg{1}{m}}} \ \nonumber
		\end{align*}
		which under \cref{asm:BoxUn} reduces to
		\begin{align}\label{eqn:RegretasCX-OptimisticCounterpart}
			\rho(\x,s)=\ & \cc_0\T\x-\inf_{\y,\u_i} \lrbr{\cc\T_0\y \colon \y\in\R^n_+,\ \u_i\in\UU^s_i,\ \sum_{i=1}^{n}y_i\left(\aa_i^0+\Ab_i\u_i \right)\leq \b_0+\Bb\u_{n+1}}.
		\end{align}
		We will first show equivalence between the above infimum and
		\begin{align}\label{eqn:ConvexReformulationofOptimisticCoutnerpart}
			\inf_{\w_0,\w_i,\u_{n+1}} \lrbr{\cc_0\T \w_0 \colon \left(w^0_i,\ \w_i\T\right)\T\in\left(\KK^s_{i}\right)^*,\ \u_{n+1}\in\UU^s_{n+1},\  \sum_{i=1}^{n}w^0_i\aa_i^0+\Ab_i\w_i \leq \b_0+\Bb\u_{n+1}},
		\end{align}
		minding the fact that by assumption, $\UU^s_i = \lrbr{\u_i \colon \left(1,\u_i\T\right)\T\in\KK^s_{i}}$ are compact.
		If $(\y,\u_1,\dots,\u_{n+1})$ is feasible for (\ref{eqn:RegretasCX-OptimisticCounterpart}), we set $w^0_i = y_i$ and $\w_i = y_i\u_i$ for $i \in \irg{1}{n}$ to get an equivalent solution for (\ref{eqn:ConvexReformulationofOptimisticCoutnerpart}). Conversely, assuming $(\w_0,\w_1,\dots,\w_n,\u_{n+1})$ is feasible for (\ref{eqn:ConvexReformulationofOptimisticCoutnerpart}), we set $y_i = w^0_i$ and $\u_i= \tfrac{\w_i}{w^0_i}$ if $w^0_i\neq0$ and $\u_i= \oo$ otherwise to obtain an equivalent solution for (\ref{eqn:RegretasCX-OptimisticCounterpart}) (here we use the fact that (\ref{eqn:0coneimplies0}) holds by assumption). By the same construction, any feasible point in (\ref{eqn:RegretasCX-OptimisticCounterpart}) with $\u_i \in \mathrm{int}\left(\UU^s_i\right)$ translates to a Slater point in (\ref{eqn:ConvexReformulationofOptimisticCoutnerpart}). So, in any case, by \cref{asm:BoxUn} d) we have strong duality with its dual given by
		$$\sup_{\z}\lrbr{\b\T\z-\tau: -\z\in \R^n_+, \left(\tau,\z\T\Bb\right)\T\in\left(\KK^s_{n+1}\right)^* , \left(c^0_i-(\aa^0_i)\T\z,\ -\z\T\Ab_i\right)\T\in \left(\KK^s_{i}\right)^* ,\mbox{ all }  i \in \irg{1}{n}}\, .$$
		This concludes the proof.
	\end{proof}
	The final supremum problem in the above proof is actually the deterministic counterpart of the "primal worst" (\ref{eqn:ColumnsWiseDual}). Thus the above theorem actually rests on  robust duality. Rather than referencing this in the proof, we chose to give a simplified, exhaustive proof, firstly for  the readers' convenience and secondly to illustrate how \cref{asm:BoxUn} ties in with the final result.

	\begin{rem}\label{rem:RemarkonNecessaryConditions}
			The necessity for the objective to be free from uncertainty stems from the fact that we require column-wise uncertainty in the max-regret function. The defining feature of column-wise uncertainty is that the $i$-th portion $\u_i$ of the uncertainty vector only interacts with the $i$-th entry of the decision variable or never interacts with any portion of the decision vector. If we introduced uncertainty into the objective this would still hold for the robust counterpart and the optimistic counterpart but not for the max-regret function. Indeed we then would have
	\begin{align*}
		\rho(\x,s)=\ & \sup_{\y,\u_i} \lrbr{\sum_{i=1}^{n}c_i(\u_i)(x_i-y_i) \colon \y\in\R^n_+,\ \u_i\in\UU^s_i,\ \sum_{i=1}^{n}y_i\left(\aa_i^0+\Ab_i\u_i \right)\leq \b_0+\Bb\u_{n+1}},
	\end{align*}
	where $\u_i$ interacts with $y_i$ and stands alone in the product $c_i(\u_i)x_i$. For our relaxation this would mean that we would have to model the interaction between $\w_i$ and $\u_i$ in (\ref{eqn:ConvexReformulationofOptimisticCoutnerpart}), since the latter variable would not disappear, and then find a convex reformulation for the resulting problem, a far more tedious task. This problem would persist even if we pushed the uncertainty into the constraints by means of an epigraphical variable.
	\end{rem}
		
		\subsubsection{RLT-based bounds}\label{sec:LP-Relax}
		
		Reformulation-linearization techniques (RLTs) are a well developed field in nonlinear optimization originally introduced in \cite{sherali1986rlt}. The idea is to generate additional constraints by multiplying existing constraints with terms of the form $\prod_{i\in \JJ_1}x_i$ and $\prod_{i\in \JJ_2}(1-x_i)$ for some appropriately chosen $\JJ_1,\JJ_2\subseteq \irg{1}{n}$ and replace the nonlinear terms with lifted variables i.e. $x_{i_1,\dots,i_n} = x_{i_1}\dots x_{i_n}$. This way one can obtain relaxations of good quality and in some cases even tight reformulations. In~\cite{sherali1990hierarchy} the authors proved that this strategy is capable of delivering hierarchies of LP-approximations for combinatorial optimization problems that eventually close the relaxation gap. However, the strategy can still be applied for more general problems, see \cite{sherali2013reformulation} for a full account.
		
		The usefulness of this approach for our purposes stems from the fact that, once a linear relaxation/reformulation is obtained, we can use LP duality in order to resolve the inner supremum defining the max-regret function and so obtain a finite reformulation. 	Speaking more formally, we construct a linear reformulation/relaxation with a  representation of the form below and a respective dual
		$$\sup_{\z\geq \oo}\lrbr{\f(\x,s)\T\z \colon \Ab\z = \b } = \inf_{\w}\lrbr{\b\T\w \colon \Ab\T\w \geq \f(\x,s)}\,. $$
		The final inequality in (\ref{eqn:REG-UPM}) can then be replaced by
		\begin{align*}
			\rho \geq \b\T\w,\ \Ab\T\w  \geq \f(\x,s)\, ,
		\end{align*}	
		as to obtain  an upper bound for the original problem, which can be tight in case the LP relaxation is exact.
		
		Though the construction of higher order approximations can be tedious, we have at least one example where a fairly simple relaxation already yields satisfactory results. We will discuss this example in the sequel of this section. Later in Section~\ref{sec:RSPP} we will apply this relaxation, also we will demonstrate in Section~\ref{sec:NumericalExperiments} that, despite its simplicity, it can yield reasonable approximation  quality.
	
	Consider the uncertain  binary linear problem
	\begin{align*}
		\inf_{\x\in\lrbr{0,1}^n}\lrbr{\cc(\u)\T\x \colon \Eb\x=\d} \quad \mbox{where }\u\in\UU\, .
	\end{align*}
	The robust counterparts of uncertain MILPs in general do not enjoy strong duality and thus we not have an easy way to resolve the inner supremum problem in (\ref{eqn:REG-UPM}) outside of special cases. In Section~\ref{sec:RSPP} we will study such cases, where the constraint matrix $\Eb\in \R^{m\times n}$ is totally unimodular. For discussion in this section we will solely focus on deriving a tractable relaxation of the max-regret function of the uncertain MILP, but already entertain the total unimodularity assumption on $\Eb$. We have
	\begin{align*}
		\rho(\x,s) = \sup_{\u\in\UU}\lrbr{\cc(\u)\T\x -\inf_{\y\in\lrbr{0,1}^n} \lrbr{\cc(\u)\T\y\colon \Eb\y=\d}}
	\end{align*}
	Assume that $\cc(\u)\coloneqq\cc_0+\Cb\u$ and that $\UU \coloneqq [-1,1]^q$. Denoting $\bar{\Cb} \coloneqq (\Cb,-\Cb)$, we have
	\begin{align*}
		\rho(\x,s)=&\sup_{\u\in[-1,1]^q,\y\in \R^n_+}\lrbr{ (\cc_0+\Cb\u)\T\x - \ (\cc_0+\Cb\u)\T\y \colon \Eb\y=\d} \\
		=&\sup_{\u\in [0,1]^{2q},\y\in \R^n_+}\lrbr{ (\cc_0+\bar{\Cb}\u)\T\x - \ (\cc_0+\bar{\Cb}\u)\T\y \colon \Eb\y=\d}.
	\end{align*}
	The first equality holds due to the total unimodularity assumption an the fact that the objective is bilinear. The second one is easily established. We see, that the optimal solutions are contained in $\{0,1\}^{n+2q}$. First, we introduce additional constraints by multiplying the linear constraints with each of the uncertainty parameters, i.e. $u_i\Eb\y  =  u_i\d$,
	all $i\in \irg{1}{q}$. We now lift the space of variables by replacing the bilinear  terms $u_iy_j$ with $\psi_{ij}$  and construct the McCormick envelope~\cite{McCormick1976} by introducing the constraints
	\begin{align*}
		\psi_{ij} \leq u_j,\ \psi_{ij} \leq y_i \, ,\
		\psi_{ij} \geq y_i+u_j-1  \, ,\mbox{ all }  j \in \irg{1}{q},\ i \in \irg{1}{n}\, .
	\end{align*}
	The end result is a  linear optimization problem the dual of which can be used  to  replace the inner supremum in (\ref{eqn:REG-UPM})  as discussed above.

	\subsection{Predictability}		
	We have constructed uncertainty regimes where the performance of the robust solution is more predictable than in others. This may be a desired property of a solution next to being robust, for example, if, after implementation of the solution, excess optimal value cannot be harnessed due to its unexpected realization (e.g. a delivery that arrives early and cannot be received). The desire to have a predictable outcome can also extend to the constraints. Being able to better estimate the resource consumption of an activity outside the worst case can be advantageous. Thus, a decision maker may prefer to operate under an uncertainty regime where the performance of the robust solution after implementation does not deviate from the robustly optimal value by too much, and where the resource consumption does not deviate too much across scenarios.	
	
	In order to have a predictable outcome, we need to reduce the gap between best-case and worst-case performance of a robust solution, arriving at the following model:
	\begin{align*}\label{eqn:PDA-UPP}\tag{PDA-UPP}
		\begin{split}
			&\inf_{\x,s,\ol{q}_i,\underline{q}_i}  \left [ \gc_s + \ga\sup_{\u_0\in \UU^s_{0}}  \lrbr{\cc(\u_0)\T\x} + \sum_{i=0}^{m}\beta_i(\ol{q}_i,\underline{q}_i)\right ] \\
			\mathrm{s.t.:}\
			&\ol{q}_0  \geq \sup_{\u_0\in\UU_0^s}\lrbr{\cc(\u_0)\T\x}, \ \underline{q}_0  \leq \inf_{\u_0\in\UU_0^s}\lrbr{\cc(\u_0)\T\x}\, ,\\
			&\ol{q}_i  \geq \sup_{\u_i\in\UU^s_i}\lrbr{\aa_i(\u_i)\T\x-b_i(\u_i)}, \ \underline{q}_i  \leq\inf_{\u_i\in\UU^s_i}\lrbr{\aa_i(\u_i)\T\x-b_i(\u_i)}\, ,\mbox{ all }  i\in \irg{1}{m}\, ,\\
			&\x  \in \RR^s_{opt}\, ,\ s \in \irg{1}{S}\, .
		\end{split}
	\end{align*}	
	Here, $\ol{q}_i$ and $\uline{q}_i$ are auxiliary epi- and hypographical variables, introduced for notational convenience only.
	 	The right-hand side suprema/infima are the worst- and best-case performances of a solution $\x$ under uncertainty regime $s$ in terms of objective function value and in terms of resource consumption. Since the differences between worst- and best-case performances are independent across rows, we may assume row-wise uncertainty, in a way similar to the robust counterpart.  Note that the suprema are convex functions in $\x$, as they are point-wise suprema of linear functions. By the same token, the infima are concave in $\x$. Thus, the respective constraints describe a convex set. The functions $\beta_i(\ol{q}_i,\uline{q}_i)$ are convex functions that put a cost on the distance between worst- and best-case performance of a feasible solution $\x$. Next to convexity we assume
	\begin{align}\label{eqn:Decreaswithdistance}
		\beta_i(\ol{q}_i,\uline{q}_i)\geq 0\, ,\quad 	\frac{\partial \beta_i(\ol{q}_i,\uline{q}_i) }{\partial \ol{q}_i} > 0\, ,\quad 	 \frac{\partial \beta(\ol{q}_i,\uline{q}_i) }{\partial \uline{q}_i} < 0 \ \mbox{for all} \  \ol{q}_i\geq\uline{q}_i\, .
	\end{align}
	A simple choice for $\beta_{i}$ would be $\beta_{i} \coloneqq \ol{q}_i-\uline{q}_i$ and in our applications we will exclusively work with this choice. However, more complicated choices are possible since in any case we will end up with a mixed-integer convex optimization problem. 	
	\begin{thm}\label{thm:Predictability}
		Problem (\ref{eqn:PDA-UPP}) is equivalent to the following mixed-integer conic optimization problem:
		\begin{align}\label{eqn:BCDModelRef}
			\begin{split}
				\inf_{\x,s,\rho_i} &\left [\gc_s+  \ga\!\left( \cc_0\T\x + \rho_0 \right)  + \beta_0(\cc_0\T\x+\rho_0,\cc_0\T\x-\gd_0)+\sum_{i=1}^{m}\beta_i\!\left( (\aa_0^i)\T\x - b_0^i +\rho_i,(\aa_0^i)\T\x -b_0^i-\gd_i\right)\right ] \\
				\mathrm{s.t.:}
				&\begin{pmatrix}
					\rho_0 ,
					-\x\T\Cb
				\end{pmatrix}\T \in \left(\KK^s_0\right)^*,\
				\begin{pmatrix}
					\gd_0,
					\x\T\Cb
				\end{pmatrix}\T \in \left(\KK^s_0\right)^*\, ,\ \x  \in \RR^s_{opt}\, ,\ s \in \irg{1}{S}\,, \\
				&\begin{pmatrix}
					\rho_i ,
					\b_i\T-\x\T\Ab_i
				\end{pmatrix}\T \in \left(\KK^s_i\right)^*, \
				\begin{pmatrix}
					\gd_i,
					\x\T\Ab_i-\b_i\T
				\end{pmatrix}\T \in \left(\KK^s_i\right)^* \, ,\mbox{ all }  i \in \irg{1}{m}\, .\\
			\end{split}		
		\end{align}
	\end{thm}
	\begin{proof}
	The proof relies on standard robust optimization arguments and is therefore relegated to the appendix.
	\end{proof}
	From the above theorem we can see that (\ref{eqn:PDA-UPP}) does not add much difficulty over solving (\ref{eqn:UPMModel}) with $\MM(\x,s)$ removed. 	 
	
	\subsection{Best-case performance}
	In \cref{exmp:ShortestPath} we see that there are choices on the uncertainty regime where the resulting robust solution exhibits good performance even in the best case. A decision maker may prefer uncertainty regimes where the conservatism inherent in the robust solution is redeemed by viable best-case performance of that solution. On the other hand, again applying this concept in terms of constraints, a decision maker may value the chance of saving resources in cases other than the worst case.	
	
	The model that can account for the best-case performance of the outcome of the uncertainty under a chosen uncertainty regime can be formulated as (using $\v$ instead of $\u$ to suggest that now uncertainty is not adverse):
	\begin{align}\label{eqn:BCP-UPP}\tag{BCP-UPP}
		\begin{split}
			\inf_{\v,\x,s,\underline{q}_i} &\left [ \gc_s + \ga\sup_{\u_0\in \UU_{0}}  \lrbr{\cc(\u_0)\T\x} + \beta\left(\underline{q}_0,\dots,\underline{q}_m\right)\right] \\
			\mathrm{s.t.:}\
			&\underline{q}_0  \geq \cc(\v)\T\x\,, \ \v \in\UU^s,\ \x \in \RR^s_{opt},\ s \in \irg{1}{S}\, \\
			&\underline{q}_i  \geq \aa_i(\v)\T\x-b_i(\v) \, ,\quad \hspace{0.05 cm}\mbox{ all }  i \in \irg{1}{m} \,.
		\end{split}
	\end{align}	
	The function $\beta(.)$ is assumed to be convex and nondecreasing in each argument; in our applications in later sections it is chosen to be the sum of $\uline{q}_i\, , \mbox{ all }i \in \irg{0}{m}$. Note that we cannot assume row-wise uncertainty, since not all rows can necessarily attain their minimum simultaneously. We would thus underestimate the optimal $\underline{q}_i $, all $i \in \irg{0}{m}$, if we were to use the projection method, that is valid for the robust counterpart (see \cite{ben2009robust}).
	The problem (\ref{eqn:BCP-UPP}) is a bilinear mixed-integer optimization problem and as such very challenging. In case $\XX_0 = \lrbr{0,1}^{n_2}$ one can again use McCormick linearization in order to obtain a problem that is manageable using C-MILP solvers.  In case that all coefficients of $\x$ depend on independent portions of the uncertainty vector $\v$ one can pre-compute and replace the uncertain coefficients with their respective minima. In the next section we will discuss an application where both of these strategies are feasible.
	
	\subsubsection{Exact reformulation in case of column-wise uncertainty}
	
	Analogous to the max-regret function, the constraints in (\ref{eqn:BCP-UPP}) can be reformulated exactly in case of column wise uncertainty. For the reformulation we do not need a duality argument so $\XX_0$ may be discrete. Also the structure of these "optimistic constraints" is simpler than the optimistic counterpart in the max-regret function, so that we can allow for uncertainty in the objective. Thus, compared to the case of max-regret, merely \cref{asm:BoxUn}  b) is needed for the reformulation discussed shortly. Note that under the assumption of column-wise uncertainty we can write the epigraphical constraints in (\ref{eqn:BCP-UPP}) as
	\begin{align*}
		\underline{q}_0 &\geq \sum_{i=1}^{n}c_i(\v_i)x_i, \quad
		\begin{pmatrix}
			\underline{q}_1,&\!\!\!\! \dots, &\!\!\!\!  \ \underline{q}_m
		\end{pmatrix}\T\geq \sum_{i=1}^{n}x_i\aa_i(\v_i)-\b(\v_{n+1}) \, .
	\end{align*}
	For the sake of notational simplicity we introduce
	\begin{align*}
		\hspace{-0.7cm}
		\bar{\aa}_i(\v_i) \coloneqq
		\begin{pmatrix}
			c_i( \v_i),&\!\!\!\! \aa_i(\v_i)\T
		\end{pmatrix}\T \, , i \in \irg{1}{n} \, , \quad
		\bar{\b}(\v_{n+1}) \coloneqq
		\begin{pmatrix}
			0,&\!\!\!\! \b(\v_{n+1})\T
		\end{pmatrix}\T \quad\mbox{and} \quad
		\underline{\q}\coloneqq \begin{pmatrix} \underline{q}_0,&\!\!\!\! \dots , &\!\!\!\! \underline{q}_m
		\end{pmatrix}\T.
	\end{align*}
	We can then find vectors $\bar{\aa}_i^0$  and matrices $\bar{\Ab}_i$
		of appropriate size such that $\bar{\aa}_i(\v_i) = \bar{\aa}_i^0+\bar{\Ab}_i\v_i$ for all $i \in \irg{1}{n}$. 
	We are now ready to state the following theorem.
	\begin{thm}\label{thm:BCreformulation}
		Assume that \cref{asm:BoxUn}  b) holds, so that (\ref{eqn:BCP-UPP}) can be written as
		\begin{align*}
			\begin{split}
				\inf_{\v_i,\x,s,\underline{\q}} &\left [ \gc_s + \ga\sup_{\u_0\in \UU_{0}}  \lrbr{\cc(\u_0)\T\x} + \beta\left(\underline{q}_0,\dots,\underline{q}_m\right)\right ] \\
				\mathrm{s.t.:}\
				&\underline{\q} \geq\sum_{i=1}^{n}x_i\bar{\aa}_i(\v_i)-\bar {\b}(\v_{n+1})\, , \ \v_i \in\UU^s_i\  \mbox{for } i \in \irg{1}{n+1},\
				\x \in \RR^s_{opt}\, ,\ s \in \irg{1}{S}\, .
			\end{split}
		\end{align*}
		Then  this problem is equivalent to
		\begin{align*}
			\begin{split}
				\inf_{\w_i,\v_i,\x,s,\underline{\q}} &\left[ \gc_s + \ga\sup_{\u_0\in \UU_{0}}  \lrbr{\cc(\u_0)\T\x} + \beta\left(\underline{q}_0,\dots,\underline{q}_m\right)\right ] \\
				\mathrm{s.t.:}\
				&\underline{\q} \geq\sum_{i=1}^{n}x_i\bar{\aa}_i^0+\bar{\Ab}_i\w_i-\bar{\b}(\v_{n+1})\, ,	\\
				&\begin{pmatrix}
					x_i,&\!\!\!\! \w_i\T
				\end{pmatrix}\T
				\in\KK^s_i \ \mbox{for } i\in\irg{1}{n}, \ \v_{n+1}\in\UU^s_{n+1}\, ,\ \x \in \RR^s_{opt}\, ,\ s \in \irg{1}{S}\, .
			\end{split}
		\end{align*}
	\end{thm}
	\begin{proof}
		The proof is similar to that of \cref{thm:WBDCwiseRef}.
	\end{proof}

	\section{The shortest path problem with endogenous arc-weight uncertainty}\label{sec:RSPP}
	In the following section, we will apply our framework to the shortest path problem.  
	As we will see, this problem can manifest in different special cases that will allow us to demonstrate the applications of both \cref{cor:CharacterizingRxB} and \cref{thm:BinrayRef2}. We want to highlight the fact that he uncertain shortest path problem appears in pre-desaster network design, which has been previously discussed in the context of robust optimization with endogenous uncertainty (see \cite{nohadani2018optimization}).

	Whenever we present the full models we represent the bilevel and the $\eps$-constraint model at the same time, by explicitly presenting the respective constraints for $\x\in\RR^s_{opt}$, and in parentheses we present the $\eps$-constraint. We will maintain this style also in the succeeding section.
	\subsection{The uncertain shortest path problem and its robust counterpart}
	Consider a directed, weighted graph $G = (\VV,\AA)$ where $\VV$ is the set of vertices and $\AA$ the set of arcs. We denote by $c_i$ the arc weight of arc $i$. Vertex $1$ is marked as \textit{source} and vertex $n$ is marked as \textit{sink}. The shortest path problem (SPP) is the combinatorial optimization problem to find a path from the source to the sink such that the sum of the weights of the arcs crossed is minimized.
	It is well known that the SPP has an exact LP reformulation
	\begin{align*}
		\min_{\x \in \R^{|\AA|}_+} \lrbr{ \cc\T\x \colon \Eb\x = \d},
	\end{align*}		
	where each (directed) arc corresponds to a variable $x_k$ under some relabeling $\pi(k)=(i,j)$ and
	\begin{align*}
		e_{i,k} =
		\left\lbrace\begin{matrix}
			-1 & \mbox{if $\pi(k) = (i,j)$ for some $j$}\\
			1 & \mbox{if $\pi(k) = (j,i)$ for some $j$}\\
			0 & \mbox{otherwise}\\
		\end{matrix}\right.
	\end{align*}
	Also $\d = (-1,0,\dots,0,1)\T\in \R^{|\VV|}$.
	That integrality can be dropped stems from the fact that $\Eb$ is totally unimodular and $\d$ is integral. The dual of the linearized SPP is given by $\max_{\y\in \R^{|\VV|}} \lrbr{\d\T\y \colon \Eb\T\y \leq \cc }$. In case of endogenous uncertain arc weights we can formulate the endogenous uncertain SPP as
	\begin{align*}\label{eqn:EU-SSP}\tag{EU-SPP}
		\min_{\x \in \lrbr{0,1}^{|\AA|}} \lrbr{ \cc(\u)\T\x \colon \Eb\x = \d}, \ \mbox{where } \u\in\UU^s,
	\end{align*}
	and the robust counterpart (for fixed $s$) is given by
	\begin{align*}\label{eqn:RSPP}\tag{R-SPP}
		\min_{\x \in \lrbr{0,1}^{|\AA|}} \lrbr{\sup_{\u\in\UU^s} \lrbr{\cc(\u)\T\x} \colon \Eb\x = \d}.
	\end{align*}
	In (\ref{eqn:RSPP}) integrality can not be relaxed in general, since the objective is no longer linear, but convex. It was proved in \cite{bertsimas2004robust} that (\ref{eqn:RSPP}) under ellipsoidal uncertainty is in fact NP-hard, and \cite{yu1998robust} showed that NP-hardness for finitely generated uncertainty sets.
	One can formulate it as a conic MILP, but then we lose the duality needed for the characterization of $\RR^s_{opt}$. In special cases, however, the robust counterpart remains a  linear optimization problem  so that we can apply \cref{cor:CharacterizingRxB}, in other cases we can still work with \cref{thm:BinrayRef2}.
	
	In the remainder of this section we discuss various reformulations and approximations of (\ref{eqn:UPMModel}) for the uncertain SPP under different cases of box uncertainty. All the formulations will involve some interaction between terms in $s$ and $\x$, and since we can model these as bilinear terms between $\r$ and $\x$ where all these variables are binary and thus between zero and one, the big-M constant for the construction of the McCormick envelope is given by 1. Also we will, in this section, push the dependence on $s$ from the uncertainty set into the description of the uncertain objective coefficients, hence we consider $\cc_s(\u)$ while the uncertainty set will remain constant across $s\in \irg 1S$.
	
	\subsection{SPP under structured box uncertainty}	
	Define $\cc_s(\u)\coloneqq \cc^0_s + \Cb_s\u$ with $\cc^0_s \in \R_+^{|\AA|}$ and $\Cb_s \in \R^{|\AA|\times q}, \ s\in\irg{1}{S}$ and consider the following uncertainty set:
	$\UU_{\mathrm{Box}} \coloneqq \lrbr{\u\in \R^q \colon \|\u\|_{\infty} \leq 1 }\,$ ,
	with $\KK =\lrbr{ (v_0,\v )\in \R^{m+1} \colon \|\v\|_{\infty} \leq v_0}$ and $\KK^* = \lrbr{ (z_0,\z)\in \R^{m+1} \colon \|\z\|_{1} \leq z_0}$.
	General box uncertainty can managed by employing \cref{thm:BinrayRef2}.	However, there are at least two cases where (\ref{eqn:RSPP}) under box uncertainty can be reformulated as a  linear optimization problem.
	\begin{enumerate}
		\item[\textit{Interval uncertainty}:] $q=|\AA|$ and $\Cb= \mathrm{Diag}(\hat\cc)$ is diagonal; by flipping $u_i\to -u_i$ if necessary, we can always assume that all $\hat c_i\geq 0$. This is a special case of what we suggest to call
		\item[\textit{Correlated interval uncertainty}]  where $\Cb_s\in  \R^{|\AA|\times q}_+$. We consider
		\begin{align*}
			\min_{\x \in \lrbr{0,1}^{|\AA|}} \lrbr{ \sup_{\u\in\UU_{\mathrm{Box}}}\lrbr{\cc_s(\u)\T\x} \colon \Eb\x = \d}\, .\end{align*}	 \end{enumerate}	
	Since $\max\limits_{\| \u \|_\infty \le 1} \g\T\u = \|\g\|_1 = \e\T\g$, with $\e\in \R^q$  the all-ones vector, holds for all $\g\in \R^{q}_+$ and since $\g=\Cb_s\T\x\in \R^q_+$ under our assumptions, we have, by total unimodularity,
	\begin{align*}
		\min_{\x \in \lrbr{0,1}^{|\AA|}} \lrbr{ \sup_{\u\in\UU_{\mathrm{Box}}}\lrbr{\cc_s(\u)\T\x} \colon \Eb\x = \d}
		=\min_{\x \in \R^{|\AA|}_+} \lrbr{ ( {\cc}^0_s + \Cb_s\e)\T\x  \colon \Eb\x = \d}\, ,\end{align*}
	which is a  linear optimization problem. It is interesting to note that in the uncorrelated interval case where $\Cb_s= \mathrm{Diag}(\hat\cc)$ with diagonal entries $\hat c_j\geq 0$ we have
	$$ \max\limits_{u\in\UU_{\mathrm{Box}}} [c^0_{s,i}+ \hat c_j u_j] = c^0_{s,j}+ \hat c_j = [\cc^0_s + \Cb_s\e]_j\quad \mbox{for all }j\in  \AA \, ,$$
	so the above  linear optimization problem  has the worst-case realizations per arc as its objective coefficients.
	Therefore in both of these cases (\ref{eqn:RSPP}) can be exactly linearized, and we can characterize the respective optimal set using strong linear duality. The reason why we split the presentation into two parts is a discrepancy in case of maximum-regret, as explained in \cref{exmp:CorrelatedIntervallUncertainty} below.
	 
	\subsubsection{SPP with interval uncertainty}\label{sec:SPIntervallUncertainty}
	In case of endogenous interval uncertainty, the travel time on an individual arc $j$ is known to lie in an interval $[\underline{c}^i_s, \ol{c}^i_s]$ if uncertainty regime $s$ is active. Thus the worst and best-case outcomes do not depend on $\x$ and are easily identified as the outcomes where the travel time is on the upper and lower bound respectively.
	Let $s \in \irg{1}{S}$ and define $\underline{\cc}_s$, and $\ol{\cc}_s$, to be the  element-wise best-case, or worst-case realizations, respectively, of $\cc(\u)$ under uncertainty regime $s$. Following above logic, we arrive at a mixed-integer linear model for each of the uncertainty preferences:
	\vspace{-0.2cm}
	\begin{align*}
		&\min_{\x,s} \gc_s + \ga \ol{\cc}_s\T\x +  \MM(\x,s) \\
		\mathrm{s.t.:}\ &\ol{\cc}_s\T\x \leq \d\T\y\, , \ \Eb\x=\d\, ,\ \Eb\T\y\leq \ol{\cc}_s,\, \left( \ol{\cc}_s\T\x \leq \varphi^{*} + \eps\right) ,\
		\x\in \R^{|\AA|}_+,\ s \in \irg{1}{S}\, .
	\end{align*}
	Here  we used \cref{cor:CharacterizingRxB} in order to characterize $\RR^s_{opt}$. To accommodate the three different types of uncertainty preferences, we specify $\MM(\x,s)$ in the following ways (with the simplest choice corresponding to the specification $\beta(\ol q,\underline q)= \beta (\overline q - \underline q)$ in~(PDA-UPP), which can be generalized in a straightforward manner):
	\begin{align*}
		\mbox{Predictability: }\MM(\x,s) =& \left(\ol{\cc}_s-\underline{\cc}_s \right)\T\x\, ,\\
		\mbox{Best-case Performance : }\MM(\x,s) =& \ \underline{\cc}_s\T\x\, ,\\
		\mbox{Regret: }\MM(\x,s) =&\ \inf_{\z}\left\lbrace\ol{\cc}_s\T\x-\d\T\z \colon \Eb\T\z \leq \cc_{\x} \right\rbrace\,,
	\end{align*}
	where $\cc_{\x} \coloneqq \underline{\cc}_s+\sum_{i=1}^{|\AA|}x_i\,\e_i\e_i\T\left (\ol{\cc}_s-\underline{\cc}_s \right )$.
	The reformulation of the regret function was introduced in \cite{karacsan2001robust}, where the concept of regret was presented under the name of {\em maximum deviation robustness.} The intuition is that, in case of interval uncertainty, the regret is maximized if the ex-post solution may only experience the best-case realization on all arcs except for those who are also crossed by the ex-ante solution.
	
	We see that the problem can be solved exactly using MILP-solvers without any further adaptation. Thus, there is at least one setting where our framework can be applied in its entirety, while never leaving the realm of basic and well-developed optimization techniques.
	
	\subsubsection{SPP with correlated interval uncertainty}\label{sec:SPP with correlated interval uncertainty}
	
	This type of uncertainty set allows us to model situations where the uncertainty across arcs is positively correlated which is irrelevant for the robust counterpart but impacts the optimistic counterpart es well as the max-regret function, because the optimistic optimizer present in both loses flexibility. In \cref{exmp:CorrelatedIntervallUncertainty} we will demonstrate the difference between interval and correlated interval uncertainty.
	
	For the uncertain travel times under uncertainty regime $s$ given by  $\cc_s(\u) \coloneqq \cc^0_s+\Cb_s\u$, let $\Cb_s$ be non-negative. Then we have  $\|\Cb_s\T\x\|_1  = \e\T\Cb_s\T\x$ so that the objective in the robust counterpart has a linear objective and the LP relaxation is still exact. Thus we can again use LP-duality in order to characterize $\RR^s_{opt}$ as follows:
	\begin{align*}
		\min_{\x,s}\  &\gc_s + \ga ({\cc}^0_s +{\Cb}_s\e)\T\x\ + \MM(\x,s)  \\
		\mathrm{s.t.:}\ &({\cc}^0_s +{\Cb}_s\e)\T\x \leq \d\T\y\, , \ \Eb\x =\d\, ,\  \x \in \R^{|\AA|}_+\, ,\ s \in \irg{1}{S} \, ,  \\
		&\Eb\T\y \leq  {\cc}^0_s +{\Cb}_s\e \,, \left( ({\cc}^0_s +{\Cb}_s\e)\T\x \leq \varphi^{*} + \eps\right)\, .
	\end{align*}
	Again, $\MM(\x,s)$ can be specified in three different ways according to the respective uncertainty preferences:
	\begin{align*}
		\mbox{Predictability: }&\MM(\x,s) = 2\e\T\Cb_s\T\x\, ,\\
		\mbox{Best-case Performance: }&\MM(\x,s) =  (\cc^0_s)\T\x - \e\T\Cb_s\T\x\, , \\
		\mbox{Regret: } &\MM(\x,s) = \b\T\w\, ,\ \Ab\T \w \geq \f(\x,s)\, ,
	\end{align*}
	where $\b,\Ab,$ and $\f$ are defined as described in Subsection~\ref{sec:LP-Relax}. We further comment on the usefulness of this bound in \cref{sec:SPPwithGeneralBoxUncertainty} and Section~\ref{sec:NumericalExperiments}. 	
	\begin{exmp}\label{exmp:CorrelatedIntervallUncertainty}
		For the best and worst-case calculation, and hence also for predictability, there is no difference between interval uncertainty and correlated interval uncertainty from a modeling standpoint. To see this, let $\Cb\in \R^{n\times m}_+$ and consider
				$$		\sup/\inf_{\|\u\|_{\infty}\leq 1}  \x\T\Cb\T\u  = \pm \x\T \Cb\T\e =
		\sup/\inf_{\|\w\|_{\infty}\leq 1} \x\T\mathrm{Diag}\left((\Cb\T\e)_1,\dots,(\Cb\T\e)_n\right)\w \, .$$
				However, there is a difference when it comes to calculating the maximum regret.
		To understand this difference, consider the graph depicted in \cref{fig:image3} and assume the traveling time on all the edges was known to be between 1 and 2. There are two paths from node 1 to node 4. Under interval uncertainty,  regardless of the path chosen, the maximum regret would be $4-2 = 2$, i.e. the worst case materializes on the path chosen, while the best case materializes for the alternative path. On the other hand, with correlated interval uncertainty we could model the situation with $\Cb = 0.5\, \e$ and $\UU = \lrbr{u \colon u\in [-1,1]}$ and $\cc_0 = 1.5\, \e$. In this case the same outcome would have to occur on all four edges simultaneously and the maximum regret would always be 0.
		\begin{figure}[ht]
			\begin{center}
				\begin{tikzpicture}
					[
					vertex/.style = {draw, inner sep = 2pt, circle},
					arc/.style = {very thick, ->},
					label/.style = {sloped, above, font=\small}
					]			
					\node (1) [vertex] at (0, 0) {$1$};
					\node (2) [vertex] at (2.5,0) {$2$};
					\node (3) [vertex] at (2.5,-1) {$3$};		 			
					\node (4) [vertex] at (5,0) {$4$};
					\draw (1) edge [arc] node [label] {} (2);
					\draw (1) edge [arc] node [label] {} (3);
					\draw (2) edge [arc] node [label] {} (4);
					\draw (3) edge [arc] node [label] {} (4);
				\end{tikzpicture}
			\end{center}
			\caption{}\label{fig:exmp2}	
			\label{fig:image3}
		\end{figure}
		
	\end{exmp}
	
	\subsection{SPP with general box uncertainty}\label{sec:SPPwithGeneralBoxUncertainty}
	If we have box uncertainty where $\Cb$ has negative entries or is not diagonal, 
	the robust counterpart of the uncertainty does no longer have an exact linearization, so that a characterization of $\RR^s _{opt}$ based on \cref{cor:CharacterizingRxB} is no longer possible, and hence \cref{thm:BinrayRef2} has to be invoked. With $\cc_s(\u)$ given as above, (\ref{eqn:UPMModel}) for the SPP with general box uncertainty can be cast as:
	
	\begin{align}
		\begin{split}\label{eqn:SPPGeneralBoxUncertainty}
			\inf_{\x,s} &\left[ \gc_s + \ga\left((\cc^0_s)\T\x+\|\Cb_s\T\x\|_1\right)  + \MM(\x,s)\right] \\
			\mathrm{s.t.:}\ &(\cc^0_s)\T\x+\|\Cb_s\T\x \|_1 \leq (\cc^0_s)\T\z_j+ \|\Cb_s\T\z_j\|_1\, ,\mbox{ all } j\in\irg{1}{S}\, ,\\
			\Eb\x &= \d \, , \left((\cc^0_s)\T\x+\|\Cb_s\T\x\|_1 \leq \varphi^*+\eps\right),\\
			\x &\in \lrbr{0,1}^{|\AA|},\ s \in \irg{1}{S}\, ,
		\end{split}
	\end{align}
	where $\z_j\in\RR_{opt}^j$ for all $j\in\irg{1}{S}$.
	The respective specifications for $\MM(\x,s)$ are
	\begin{align*}
		\mbox{Predictability: }&\MM(\x,s) = 2\|\Cb_s\T\x\|_1\, ,\\
		\mbox{Best-case performance: }&\MM(\x,s) = \inf_{\u\in \UU}\lrbr{(\cc^0_s)\T\x+\u\T\Cb_s\T\x}\, ,\\
		\mbox{Regret: } &\MM(\x,s) = \b\T\w\, ,\ \Ab\T\w \geq \f(\x,s)\, .
	\end{align*}
	Since $\x$ and $s$ are discrete, we can linearize the multilinear term  in the infimum using iteratively McCormick envelopes. For the max-regret calculation we have to work again with the approximations described in Subsection~\ref{sec:Regret}. However, when we observe the max-regret function for uncertain SPP given by
	\begin{align*}
		\rho(\x,s) = \sup_{\u\in \UU}\lrbr{\cc_s(\u)\T\x -\inf_{\y\in\lrbr{0,1}^{|\AA|} } \lrbr{\cc_s(\u)\T\y \colon  \Eb\y = \d}}
	\end{align*}
	we see that we have a mixed-binary problem where the bilinear terms can be exactly linearized using McCormick envelopes.  This allows us at least to easily test the quality of our proposed lower upper bound. The same is true for the regret in the correlated interval uncertainty model. We will report a respective evaluation in Section~\ref{sec:NumericalExperiments}.
	
	As stated above, the constraints enforcing $\x\in\RR^s_{opt}$ are derived using \cref{thm:BinrayRef2}. The number of these constraints may be exponential (as discussed above) and each requires solving an optimization problem in order to determine the right-hand side. We circumvene these problems in our implementation by applying these constraints as lazy constraints. Thus, in our implementation described in the next section, the constraints are absent at first, and whenever the branch-and-bound solver encounters a feasible solution $(\x,s)$, it checks whether
	\begin{align*}
		(\cc^0_s)\T\x+\|\Cb_s\T\x \|_1 \leq \inf_{\y\in\lrbr{0,1}^{|\AA|}} \lrbr{\sup_{\u\in\UU}\lrbr{\cc_s(\u)\T\y }\colon  \Eb\y = \d}.
	\end{align*}
	If this is not the case, we add the constraint to the model. This process proved to be quite efficient since the robust optimization problem on the right-hand side is a MILP that is quite easily tackled by modern solvers.

	\section{Numerical experiments}\label{sec:NumericalExperiments}
	In this section we will conduct some numerical experiments, where we will apply the concepts introduced in this text to the shortest path problem. We will demonstrate that we can use general MILP solvers in order to solve (\ref{eqn:UPMModel}) in this setting for interval uncertainty and general box uncertainty. Also we will investigate the trade-off between the different uncertainty preferences and worst-case performance. All instances were solved using {\tt Gurobi}~\cite{gurobi}, and the implementation was done in {\tt Python}~\cite{python} and {\tt Matlab}. The experiments were run using an AMD Ryzen 7 1700X Eight-Core 3.4 GHz CPU, and 16 GB RAM.
	
		\subsection{The $\eps $-constraint model for the SPP}
		
		In this set of experiments we exemplify the use of the $\eps $-constraint model when applied to the uncertain shortest path model under interval uncertainty. To this end, we created 10 instances of these problems by generating networks of size $|\VV| = 150$, where nodes $i$ and $j$ where connected with 50\% probability if the indexes of the nodes where no further apart than 10 and with probability equal to zero otherwise. For each instances we created $S = 10$ uncertainty sets by generating uncertainty arc-weights where the lower bound were drawn at random uniformly from the interval $ [8,10]$ and upper bounds were drawn from the interval $[10,12]$. For each instance we solved the endogenously robust counterpart in order to obtain
		$\varphi^*$. Replacing $\MM(\x,s)$ with the respective terms for predictability, best-case performance and regret one at a time, we resolved each instance allowing for an $\eps $ violation of the optimality constraint where $\eps \in \lrbr{1\%,2\%,\dots, 10\%}$. The blue line in the graphs
 of Figure~3 depicts the change of the optimal solutions predictability, best-case performance and max regret with $\eps $, while the orange line depicts the change of the solutions worst case performance. The values were normalized by their range across different values of $\eps $ and averaged across all 10 instances.
		
\begin{figure}[h!]\label{fig:epsilonmodel}
			\centering
			\includegraphics[scale= 0.38]{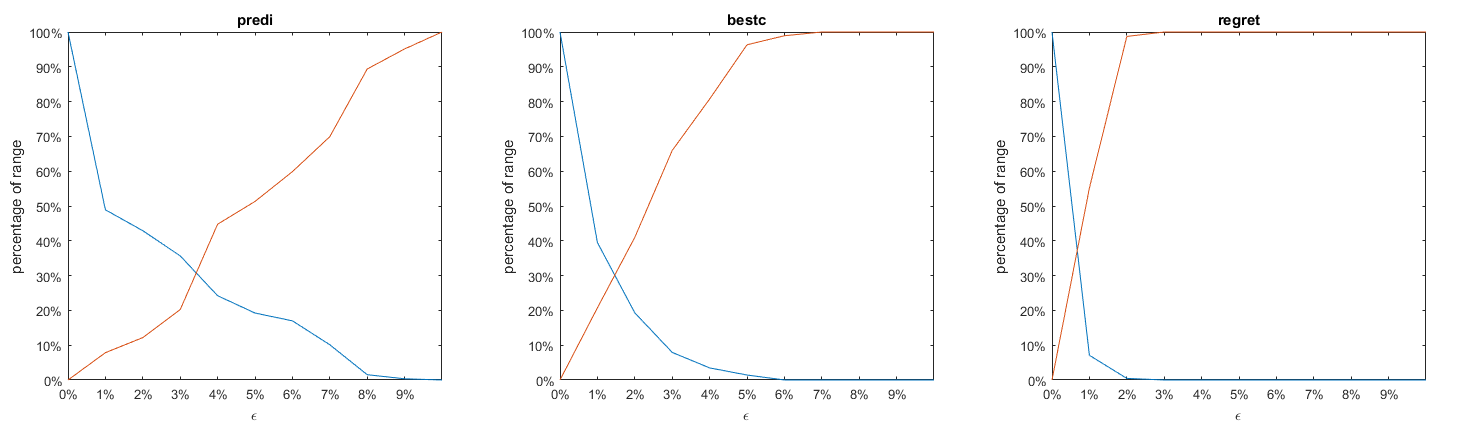}	
						\caption{Trade-off between the secondary characteristics and robustly optimal value.}	
\end{figure}
				
		We see that in our experiments on average, allowing for a $1\%$ deviation from the optimal value earned a $50\%$ gain in predictability relative to the performance of the robustly optimal solution. A similar outcome is observed when optimizing best-case performance, albeit at a larger sacrifice of worst case performance relative to the range. With regard to regret, the trade-off was even more pronounced. We also report the run time for these experiments in~\cref{tbl:epsilonRunningTime}.
		\begin{table}[h!]
			\centering	
			\begin{tabular}{c| c c  }
				preferences& $T$  & [$T_{min}$, $T_{max}$]  \%  \\
				\hline
				\hline
				predi  & 53.69 &[12.43, 200.12]  \\
				
				bestc  & 42.83  &[10.05, 174.17]   \\
				
				regret & 592.29  &[1388.64, 328.43]  \\
			\end{tabular}
			\caption{Run time for the $\eps$-constraints model}\label{tbl:epsilonRunningTime}
		\end{table}

		Of course the results above cannot be generalized and depending on the instances at hand one could face much more favourable trade-offs. However, these instances at least demonstrate that such favourable trade-offs where small deviations from robust optimality allow for substantial gains in the secondary objective are possible.	 This also demonstrates how our theory can be applied in practice, as it gives decision makers an idea of how much of protection against worst case performance they have to sacrifice, in order to make a meaningful difference with respect to predictability, best-case performance, or regret.
	
	\subsection{The bilevel model for the SPP}
	For these experiments we considered the uncertain SPP under two types of endogenous box uncertainty: interval uncertainty and general box uncertainty. For both types we want to illustrate the trade-off between different contributors of the objective function of (\ref{eqn:UPMModel}), namely worst-case performance, predictability, maximum regret and best-case performance. We compared the trade-off between two of these at a time so that in total 6 comparisons were made for both types of uncertainty models. The attributes not considered were dropped from the objective function, while the two that were considered had randomly chosen multipliers from the unit box. We probed 14 random coefficients per instance and generated 20 instances per pair of attributes so that in total we solved 280 problems per pair. The size of the instances was $|\VV|=300,\ S = 10$ for the interval uncertainty model and $|\VV|=50,\ S = 10,\ q = 10$ for  the general uncertainty model. Arcs between nodes were established randomly, where an arc between nodes $i$ and $j$ was created if a uniformly distributed random number from the interval $[0,1]$ was smaller than $0.3/|i-j|$. For the general box uncertainty we used the upper bound from Subsection~\ref{sec:Regret} for the computation of regret. However, for each solution $(\x,s)$ we obtained, we also calculated $\rho(\x,s)$ exactly and compared that value to the value from the upper bound. In 1678 of 1680 cases the bound was in fact tight, indicating that it worked reasonably well for problems of the size considered here. In our experience the quality of the bound diminishes as $|\VV|$ and $|\AA|$ grow but for medium size problems it is quite acceptable.
	
	We summarize the results in \cref{tbl:GeneralUncertainty} and \cref{tbl:IntervalUncertainty}. The tables report the average computation time in seconds ($T$) as well as the range of computation times ([$T_{min}$, $T_{max}$]). The column ($\#TO$) gives a  measure of how intensive the trade-off between different attributes has been in our experiments. It reports the number of instances for which the value of the compared attributes changed with different objective function coefficients. The columns   ($\Delta_1$\%) and ($\Delta_2$\%) give the percentage difference between the maximum and the minimum value of the first and second attribute within an instance, averaged over all 20 instances.
	
	Further, \cref{fig:GeneralUncertaintyScatter} and \cref{fig:IntervallUncertaintyScatter} illustrate what can be thought of as combined, normalized efficiency frontiers. For every problem solved, we normalized the outcome of the attributes on a scale between 0 and 1 using the formula $a_{\%} = (a-a_{min})/(a_{max}-a_{min})$, where $a_{min}$ and $a_{max}$ are the minimum and the maximum the attribute attained in that instance over all 14 probings, and $a$ is the unnormalized value of that attribute in the current problem. After normalizing one can plot the 14 pairs of outcomes for the two attributes being compared in a scatter plot. What is depicted in the figures is the combination of all of these plots per pair. We only included those instances where there was a trade-off between the two attributes.
	
	We summarize the key takeaways from our experiments as follows: Due to the proposed formulations, even the most demanding model with general box uncertainty is manageable using standard software. Solutions were mostly obtained in a matter of minutes. Trade-offs between two attributes were encountered regularly, indicating that different uncertainty preferences have a notable impact on the decision. In instances where there was a trade-off, outcomes of the attributes could vary substantially across the 14 probings.
	
	Our experiments also demonstrate how the bilevel model can be used in practice. The top level decision maker (who controls the uncertainty set) could base their decision on efficiency frontiers displaying the trade-off between different attributes across the different robustly optimal sets and implement the uncertainty regime that optimizes that trade-off.

	\begin{figure}[h!]
		\centering
		\includegraphics[scale= 0.40]{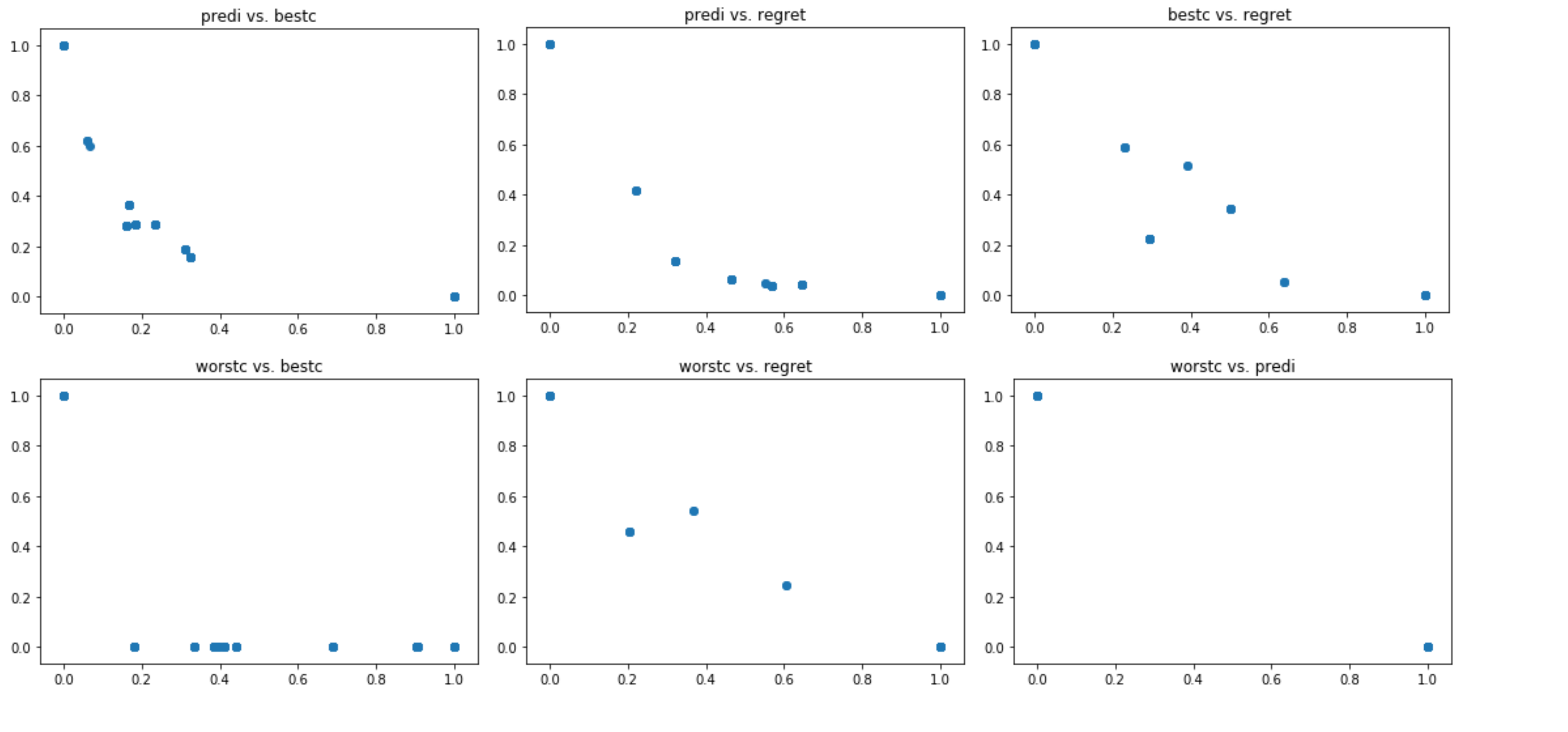}	
		\caption{Scatter plots for SPP instances with general box uncertainty}
		\label{fig:GeneralUncertaintyScatter}	
	\end{figure}
	\begin{table}[h!]
		\centering	
		\begin{tabular}{c| c c c c c c }
			preferences& $T$  & [$T_{min}$, $T_{max}$]  & $\#TO$  & $\Delta_1$\%    & $\Delta_2$ \%  \\
			\hline
			\hline
			predi vs. bestc   & 131.63 &[26.16, 418.12] &20  &61.81  &16.81 \\
			
			predi vs. regret  & 18.48  &[4.51, 72.24]   &9   &16.15  &30.38 \\
			
			bestc vs. regret  & 115.77 &[9.75, 328.43]  &9   &3.21   &20.35 \\
			
			worstc vs. bestc  &91.82   &[3.57, 318.92]  &20  &14.38  &13.03 \\
			
			worstc vs. regret &13.68   &[5.63, 43.56]   &8   &1.64   &18.04 \\
			
			worstc vs predi  &5.04    &[1.66, 16.14]   &5   &2.37   &9.39   \\  		
		\end{tabular}
		\caption{Results for instances with general box uncertainty}\label{tbl:GeneralUncertainty}
	\end{table}

	\begin{figure}[h!]
		\centering
		\includegraphics[scale= 0.40]{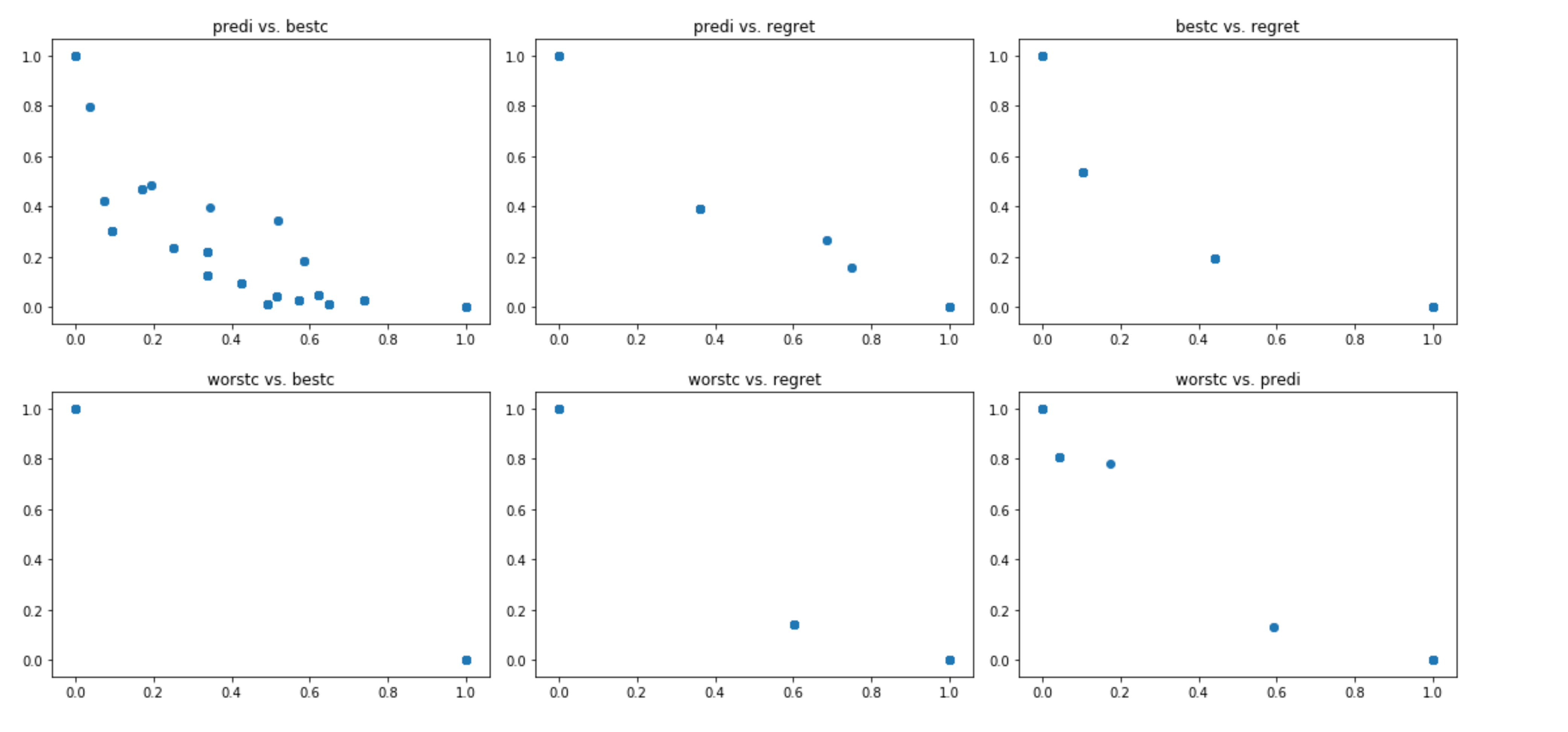}	
		\caption{Scatter plots for SPP instances with interval uncertainty}
		\label{fig:IntervallUncertaintyScatter}	
	\end{figure}	
	\begin{table}[h!]
		\centering
		\begin{tabular}{c| c c c c c }
			preferences& $T$  & [$T_{min}$, $T_{max}$]  & $\#TO$  & $\Delta_1$\%    & $\Delta_2$ \%   \\
			\hline
			\hline
			predi vs. bestc   &14.89   &[7.76, 28.7]    &19  &44.87  &42.95  \\
			
			predi vs. regret  &24.8    &[7.48, 85.56]   &15  &21.9   &23.91    \\
			
			bestc vs. regret  &22.35   &[9.24, 75.66]   &9   &7.25   &19.73 \\
			
			worstc vs. bestc  &14.0    &[7.56, 47.97]   &9   &4.12   &11.29  \\
			
			worstc vs. regret &20.88   &[8.44, 36.87]   &3   &2.46   &6.86     \\
			
			worstc vs predi   &15.44   &[6.8, 55.16]    &12  &7.12   &15.78  \\
		\end{tabular}
		\caption{Results for SPP instances with interval uncertainty}\label{tbl:IntervalUncertainty}
	\end{table}
	\newpage
		
	\section{Conclusion}
		We have introduced a two models, based on multi-objective optimization and bilevel optimization, for optimization under endogenous uncertainty that capture uncertainty preferences of a decision maker, where the aim is to choose an uncertainty regime such that the robustly optimal solutions under said regime exhibits preferred qualities with respect to predictability, maximum regret and best-case performance. Along with introducing the -- apparently novel -- models, we also discussed, for relevant special cases, various reformulations that can be tackled using C-MILP solvers. We applied this framework to the shortest path problem and (in the appendix) to the knapsack problem under endogenous uncertainty, providing multiple practical reformulations applicable under different uncertainty models. We demonstrated that the respective optimization problems are tractable, using standard software, and give meaningful results. However, we understand that there are many open questions regarding this new model, some of which we discuss briefly in the appendix.
		
\noindent
{\bf Acknowledgements.}
The authors want to thank two anonymous referees and an anonymous Associate Editor for their advice and suggestions to improve upon earlier versions.
	
\setlength{\bibsep}{-2.6pt plus 0.0ex}
		\bibliography{Literature}
	\bibliographystyle{abbrv}
	
		\appendix
	
	\newpage
	
	\hspace{-1cm}\Huge{Appendix} \\
	\normalsize
	
	\noindent

	\section{Generalization from SPP 
		to other network flow problems}\label{apx:Generalization of our discussion on the shortest path problem to other network flow problems}
	
	Network flow problems problems take the form of $\inf_{\x}\lrbr{\cc\T\x\colon \Eb\x\leq \b,\ \x\in\lrbr{0,1}^{|\AA|}}$ where the matrix $\Eb\in\R^{q\times|\AA|}$ is parametrized by the underlying graph, $\b\in\R^{q}$ is a vector of maximal or minimal net flows at each node and $\cc$ is a vector of costs associated with flows on each arc. Note, that in the aforementioned general form we have assumed that possible "$\geq$"-constraints have been transformed by multiplying both sides by $-1$. Network flow problems are, thus, combinatorial problems which in many special cases allow for an exact LP relaxation. Specifically, if $\Eb$ is totally unimodular and $\b$ is a vector of integers,  than the LP relaxation is actually tight. We will now discuss the assumptions under which the above techniques can be applied to the more general case.
	
	If we assume interval uncertainty of the objective function coefficients, and furthermore that the LP relaxation of the nominal problem is tight, then the entire discussion in \cref{sec:SPIntervallUncertainty} applies almost verbatim. Specifically, this is true for the max-regret reformulation as we will show shortly in \cref{thm:MaxRegretRef}. Also, if we assume that $\b$ is uncertain, no adoptions are needed for the modelling of predictability and best-case performance in the objective, since the latter is independent of $\b$. However, in order to keep \cref{cor:CharacterizingRxB} applicable we need to assume that $\b$ is also subject to interval uncertainty: $b_i \in [\uline{b}_i,\bar{b}_i]$,
	$ i \in \irg{1}{q}$, with integer $\uline{b}_i,\bar{b}_i$. The robust counterpart is then  $\inf_{\x}\lrbr{\uline{\cc}\T\x\colon \Eb\x\leq \bar{\b},\ \x\in[0,1]^{|\AA|}}$, where $\bar{\b}$ is the (integer) vector of worst case realization of the uncertain right hand side. By totally unimodularity we can cast the latter problem as an LP so that we may use \cref{cor:CharacterizingRxB} in order characterize $\RR_{opt}^s$. The same assumption is needed in order to maintain  validity of the reformulation of the regret function in \cref{sec:SPIntervallUncertainty} (up to minor adaptations) which we will demonstrate in \cref{thm:MaxRegretRef} at the end of this section.
	
	For correlated box uncertainty with fixed right hand side. the discussion in \cref{sec:SPP with correlated interval uncertainty}  applies analogously, however the relaxation used to model the max-regret function needs to be adapted using the more general techniques discussed at the beginning of \cref{sec:LP-Relax}. For all other cases, e.g. general box uncertainty and correlated box uncertainty with uncertain right hand side, one can apply the techniques from \cref{sec:SPPwithGeneralBoxUncertainty}, again with the need to adapt the  max-regret relaxation.
	
	We close the section providing a
	slight
	generalization of the reformulation of the max-regret function from \cite{karacsan2001robust}, which covers more general network flow problems and cases with right-hand side uncertainty.
	\begin{thm}\label{thm:MaxRegretRef}
		Assume that for the uncertain linear-binary problem
		\begin{align*}
			\inf_{\x}\lrbr{\cc(\u)\T\x\colon \Eb\x\leq \b(\v),\ \x\in\lrbr{0,1}^{|\AA|}} \mbox{ where } \|\u\|_{\infty}\leq 1,\ \|\v\|_{\infty}\leq 1
		\end{align*}
		we have interval uncertainty for both the objective and the right hand sides, i.e. $(\cc(\u))_i = c_i(u_i) \in [\underline{c}_i,\bar{c}_i]$ and $(\b(\v))_i = b_i(v_i) \in [\underline{b}_i,\bar{b}_i]$ where $\underline{b}_i,\bar{b}_i$ are integer, $i \in \irg{1}{q}$, and that $\Eb$ is totally unimodular. Also, let $\cc_{\x}$ be defined as in \cref{sec:SPIntervallUncertainty}. Then for any $\x\in\lrbr{0,1}^{|\AA|}$ the max-regret function is given by
		\begin{align*}
			\rho(\x) = \min_{\z\in\R^{q+|\AA|}}\lrbr{\bar{\cc}\T\x- \left(\uline{\b}\T,\e\T\right)\z\colon \left(\Eb,\Ib\right)\T\z\leq\cc_{\x}}.
		\end{align*}
	\end{thm}
	\begin{proof}[Proof of \cref{thm:MaxRegretRef}]
		\begin{align*}
			\rho(\x) &= \max_{\u,\y,\v}\lrbr{\cc(\u)\T\x - \cc(\u)\T\y \colon \Eb\y\leq \b(\v),\ \y\in\lrbr{0,1}^{|\AA|},\ \|\u\|_{\infty}\leq 1,\ \|\v\|_{\infty}\leq 1}\\
			&= \bar{\cc}\T\x - \min_{\y,\v}\lrbr{\cc_{\x}\T\y \colon \Eb\y\leq\b(\v),\ \y\in\lrbr{0,1}^{|\AA|},\ \|\v\|_{\infty}\leq 1}\\
			&=  \bar{\cc}\T\x - \min_{\y,\v}\lrbr{\cc_{\x}\T\y \colon \Eb\y\leq\underline{\b},\ \y\in\lrbr{0,1}^{|\AA|}}\\
			&=  \bar{\cc}\T\x - \min_{\y,\v}\lrbr{\cc_{\x}\T\y \colon \Eb\y\leq\underline{\b},\ \y\in[0,1]^{|\AA|}}\\
			&=  \bar{\cc}\T\x - \max_{\z\R^{q+|\AA|}}\lrbr{\left(\uline{\b}\T,\e\T\right)\z\colon \left(\Eb,\Ib\right)\T\z\leq\cc_{\x}}
		\end{align*}
		The second equality stems from the fact that for any $\x,\y\in\lrbr{0,1}^{|\AA|}$ we have
		\begin{align*}
			\max_{\|\u\|_{\infty}\leq 1}\lrbr{\cc(\u)\T (\x-\y)} &= \max_{\|\u\|_{\infty}\leq 1}\lrbr{\sum_{i=1}^{|\AA|}c_i(u_i)(x_i-y_i)} \\
			&= \sum_{i\in\II_1} \max_{-1\leq u_i\leq 1}\lrbr{c_i(u_i)(x_i-y_i)} + \sum_{i\in\II_0} \max_{-1\leq u_i\leq 1}\lrbr{c_i(u_i)(x_i-y_i)}\\
			&= \sum_{i\in\II_1} \bar{c}_i(x_i-y_i) +\sum_{i\in\II_0} \underline{c}_i(x_i-y_i)\\
			&= \sum_{i\in\II_1} \bar{c}_i(x_i-y_i) -\sum_{i\in\II_0} \underline{c}_iy_i = \sum_{i=1}^{|\AA|}\bar{c}_ix_i +  \sum_{i\in\II_1} \bar{c}_iy_i -\sum_{i\in\II_0} \underline{c}_iy_i\\
			&= \bar{\cc}\T\x + \cc_{\x}\T\y,
		\end{align*}
		where $\II_j = \lrbr{i\colon x_i = j},\ j = 0,1$. The third equality is justified since it is optimal to make the constraints as loose as possible, the forth stems from the fact that $\Eb$ is totally unimodular and $\underline{\b}$ is integer and, finally, the fifth equality holds by linear programming duality.
	\end{proof}

	\section{Knapsack problems with endogenous $\Gamma$-uncertainty}\label{apx:The knapsack problem with endogenous Gamma-uncertainty}
	In this section we will apply our framework to the knapsack problem with $\Gamma$-uncertainty which is a classical example from \cite{bertsimas2004price}. In this setting, we want to maximize the utility of packed items under the constraint that their joint weight must not surpass a certain threshold.
	Let $\cc\in\R^n$ be the vector of utilities of the $n$ items, and let $\aa\in\R^n$ their weight, then the knapsack problem is given by $\max_{\x\in\lrbr{0,1}^{n}} \lrbr{\cc\T\x \colon  \aa\T\x \leq b}$.
	As in \cite{bertsimas2004price} we assume that the weights of the individual items are uncertain, where the number of items whose wight deviates from the nominal value $\aa_0\in\R^n$ is restricted by the parameter $\Gamma\in\N$. Denoting the maximum amount of deviation of item $i$ by $\bar{a}_i, \ i \in\irg{1}{S}$ our assumptions so far may be formalized as $(\aa(\u))_i\coloneqq a^0_i + \bar{a}_iu_i$ where $\u\in \UU_{\Gamma}\coloneqq \lrbr{\u\colon \sum_{i=1}^{n}u_i  = \Gamma,\ \u\in[0,1]^n}$. However, we make the additional assumption that we can influence the vector of maximum deviations by makeing a choice on the uncertainty regime, that is we can chose between $\bar{\aa}_s\in\R^n,\ i \in\irg{1}{S}$. Under these assumptions, the knapsack problem under endogenous $\Gamma$-uncertainty is given by
	\begin{align*}
		\max_{\x\in\lrbr{0,1}^{n}} \lrbr{\cc\T\x \colon \sum_{i=1}^{n}(a^0_i+\bar{a}^s_iu_i) x_i \leq b} \quad \mbox{where } \u\in\UU_{\Gamma}, \ s\in\irg{1}{S},
	\end{align*}
	while the endogenously robust counterpart is given by
	\begin{align*}
		\max_{\x\in\lrbr{0,1}^{n},\ s\in\irg{1}{S}} \lrbr{\cc\T\x \colon  \aa\T\x + \max_{\u}\lrbr{ \sum_{i=1}^{n} u_i\bar{a}^s_ix_i\colon \sum_{i=1}^{n}u_i  = \Gamma,\ \u\in[0,1]^n } \leq b}.
	\end{align*}
	Following the standard reformulation strategy from \cite{bertsimas2004price} we arrive at the deterministic counterpart given by
	\begin{align*}
		\max_{\x,\gl,\ggd,s} \ \cc\T\x  \hspace{2cm}& \\
		\mathrm{s.t.:}\ \aa\T\x + \gl\Gamma +\ggd\T\e&\leq b,\ \x\in\lrbr{0,1}^{n},\ s\in\irg{1}{S},\ \ggd\in\R^n_+,\ \gl \in \R\,, \\
		\gd_i +\gl &\geq \bar{a}^s_ix_i ,  \hspace{1cm} i \in \irg{1}{n}\,.
	\end{align*}
	Again we pushed the dependence on $s$ into the coefficients $\bar{a}_i^s$ and the bilinear terms arising from $\bar{a}^s_ix_i = \sum_{j=1}^{S}\bar{a}^j_ir^s_jx_i$ can be linearized exactly via McCormick envelopes where the arising big-M constants are equal to one.
	
	When applying the bilevel optimization framework from \cref{sec:A bilevel optimization approach} we have to resort to \cref{thm:BinrayRef1} for the characterization of $\RR_{opt}^s$ since we have a combinatorial problem with endogenously uncertain constraints. This raises the question of how to construct $M_j, \ j \in \irg{1}{S}$, for the respective constraint from (\ref{eqn:BinrayRef}). In case of the endogenously uncertain knapsack problem this constraint is given by $\cc\T\x\geq\varphi_j+(1-r^s_j)M_j, \mbox{ all } j\in\irg{1}{S}$, where we had to flip the inequality since we are now dealing with a maximization problem and we need $M_j, \ j \in \irg{1}{S}$ to be a upper bound on $\varphi_{\min}-\varphi_j,\ j \in \irg{1}{S}$ with $\varphi_{\min}\coloneqq \min_{j\in\irg{1}{S}}\varphi_j$. Although we discussed a general scheme for constructing these constants in \cref{sec:A bilevel optimization approach}, we can use special structure to obtain cheaper bounds on $\varphi_{\min}-\varphi_j,\ j \in \irg{1}{S}$. Note that $M_j$ must be such that the constraint becomes redundant whenever $r^s_j=0$. This is achieved by choosing $M_j = \varphi_j$ since zero is a trivial lower bound for the objective of the knapsack problem. It is therefore not necessary to pre-calculate $\varphi_{\min}$, but one also has to take into account that the weaker bound will result in weaker relaxations at each node of a branch and bound procedure which may adversely affect the run times of these algorithms.
	
	We will now proceed to discuss the various specifications of $\MM(\x,s)$ for the different types of uncertainty preferences and comment on some of the special issues to be considered when working with these models.
	
	\subsection{Regret}
	As discussed in \cref{sec:LP-Relax} we can always find exact linearizations of the max-regret function at the cost of introducing an exponential number of variables. For the sake of simplicity we will limit our discussion to a compact linearization based on relaxing the McCormick reformulation and dualization of the latter. The detailed steps are outlined in the following theorem:
	
	\begin{thm}\label{thm:MaxRegretforKnapsack}
		Assume that there is a value for the uncertain parameter such that the knapsack problem under endogenous $\Gamma$-uncertainty is feasible. Then, the max-regret function can be upper bounded in the following way
		\begin{align}\label{eqn:KnapssackRegretReformulation}
			\rho(\x,s) \leq  -\cc\T\x + \min_{\mu,\gga,\ggb,\ggc,\gl}\lrbr{\mu b+ \gl\Gamma+ \e\T(\ggc+\t+\s)\colon
				\begin{array}{rl}
					\aa\mu - \ggb +\ggc+\t \geq& \hspace{-0.2cm}\cc,\\
					\bar{\aa}^s\mu + \gga +\ggb -\ggc \geq&\hspace{-0.2cm} \oo,\\
					\e\gl - \gga + \ggc + \s \geq&\hspace{-0.2cm} \oo,\\
					\ \gga,\ggb,\ggc,\t,\s\in&\hspace{-0.2cm} \R^n_+, \ \mu\geq 0
				\end{array}
			}	
		\end{align}
	\end{thm}
	\begin{proof}
		We have
		\begin{align}
			\begin{split}
				\hspace{-0.5cm}\rho(\x,s) =& \max_{\y,\z} \lrbr{\cc\T\y-\cc\T\x \colon \aa\T\y + \sum_{i=1}^{n} z_i\bar{a}^s_iy_i \leq b,\ \sum_{i=1}^{n}z_i  = \Gamma,\ \z\in[0,1]^n , \y\in\lrbr{0,1}^n }\\
				=& -\cc\T\x + \max_{\y,\z,\u} \lrbr{\cc\T\y \colon
					\begin{array}{l}
						\aa\T\y + \sum_{i=1}^{n} \bar{a}^s_iu_i \leq b,\\
						\sum_{i=1}^{n}z_i  = \Gamma, \ \u \geq \z+\y-\e,\ \u\leq \z,\ \u\leq \y,\\
						\z\in[0,1]^n , \y\in\lrbr{0,1}^n, \u \geq \oo
				\end{array}}\\
				\leq& -\cc\T\x + \max_{\y,\z,\u} \lrbr{\cc\T\y \colon
					\begin{array}{l}
						\aa\T\y + \sum_{i=1}^{n} \bar{a}^s_iu_i \leq b,\\
						\sum_{i=1}^{n}z_i  = \Gamma, \  \u \geq \z+\y-\e,\ \u\leq \z,\ \u\leq \y,\\
						\z\in[0,1]^n , \y\in[0,1]^n,\u \geq \oo
				\end{array}}\\
				=& -\cc\T\x + \min_{\mu,\gga,\ggb,\ggc,\gl}\lrbr{\mu b+ \gl\Gamma+ \e\T(\ggc+\t+\s)\colon
					\begin{array}{rl}
						\aa\mu - \ggb +\ggc+\t \geq& \hspace{-0.2cm}\cc,\\
						\bar{\aa}^s\mu + \gga +\ggb -\ggc \geq&\hspace{-0.2cm} \oo,\\
						\e\gl - \gga + \ggc + \s \geq&\hspace{-0.2cm} \oo,\\
						\ \gga,\ggb,\ggc,\t,\s\in&\hspace{-0.2cm} \R^n_+, \ \mu\geq 0
					\end{array}
				},
			\end{split}
		\end{align}
		where the final step follows from linear duality and the fact that we consider a feasible knapsack problem.
	\end{proof}
	
	This relaxation may be primitive but in our experiments showed to be off by no more than 2\% at most for the instances we considered. Of course, such a difference could potentially be troubling in case different uncertainty regimes give similar regret, but are affected differently by the relaxation error. For the sake of simplicity, however, we were willing to make this sacrifice in our experiments.
	
	A notable problem that arises from the above reformulation stems from the term $\bar{a}^s_i\mu = \sum_{i=1}^{S}\bar{a}_i^sr^s_i\mu$ which is bilinear in $\r_s$ and $\mu$. Of course, since $\r_s$ is a binary vector we can achieve an exact McCormick linearization by considering,
	\begin{align*}
		\bar{a}^s_i\mu = \sum_{j=1}^{S}\bar{a}_i^sr^s_j\mu = \sum_{i=1}^{S}\bar{a}_i^sw_i\,, \quad w_i = r^s_i\mu\,, \ \mbox{ all } i \in \irg{1}{S}
	\end{align*}
	where the latter equality is replaced by its McCormick envelope given by
	\begin{align*}
		w_i\leq r^s_i \ol{M}_i\,,\
		w_i\geq 0\,,\
		w_i\leq \mu\,,\
		w_i\geq \mu-(1-r^s_i)\ol{M}_{max}\, , \ \mbox{ all } i \in \irg{1}{S} ,
	\end{align*}	
	where $\ol{M}_i,\ i\in\irg{1}{S}$ are upper bounds on $\mu$ under uncertainty regime $i\in\irg{1}{S}$, while $ \ol{M}_{max} = \max\lrbr{\ol{M_i} \colon i \in \irg{1}{S}}$. This construction is necessary since the final constraint must be redundant for all but the activated uncertainty regimes, while the first constraints need to be redundant only for the activated uncertainty regime. The difficulty is the construction of valid upper bounds which we provide in the following theorem:
	
	\begin{thm}\label{thm:BoundforMu}
		Let $s\in\irg{1}{S}$ be fixed and consider the minimization problem in (\ref{eqn:KnapssackRegretReformulation}) with $\Gamma<n$. Then for any optimal choice of $\mu^*$ we have $\mu^* \leq \ol{M}_i = \max\lrbr{ \frac{c_j}{a_j}\colon j\in\irg{1}{n} }$, and this bound cannot be improved upon.
	\end{thm}
	\begin{proof}
		We will argue that the statement is true for any positive $b< a_1$ and conclude the proof by showing that this gives an upper bound for $\mu^*$  for all other positive $b$. So choose numbers $b<a_1$ and $\Delta b>0$  such that still $b+\Delta b < a_1$. We can reorder indices $j$ such that  $\frac{c_1}{a_1} \geq \frac{c_j}{a_j}$ for all $ j\in\irg{1}{n}$. Then the last maximization problem in (\ref{eqn:KnapssackRegretReformulation}) has as a  feasible solution
		$\y = \frac b{a_1}\, \e_1$, $\z=\sum_{j=n-\Gamma}^n \e_j$, $\u =\oo$, with objective function value $\frac{bc_1}{a_1}$. The problem $\max_{\y\geq 0}\lrbr{\cc\T\y\colon \aa\T\y\leq b} = \frac{bc_1}{a_1}$ is a relaxation, hence the aforementioned feasible solution is actually optimal. Replacing $b$ by $b+\Delta b$ does not change
		above construction whenever the right-hand side $b+\Delta b$ stays between zero and $a_1$. Therefore the dual solution stays the same, as it is feasible for both variants, and complementary slackness continues to hold (only $y_1$ changes and the new constraint $\aa\T\y\leq b+\Delta b$ is again binding at $\y$). Thus, for the optimal choice $\mu^*$,
		the difference of dual optimal values (by strong duality, equalling the difference of the primal-optimal values) is
		$$
		\mu^*\Delta b  = \max_{\y,\u,\z}\lrbr{\cc\T\y\colon \aa\T\y+\aa_s\T\u\leq b+\Delta b, \dots} - \max_{\y,\u,\z}\lrbr{\cc\T\y\colon \aa\T\y+\aa_s\T\u\leq b, \dots}  =\frac{ c_1\Delta b}{a_1}
		$$
		so that $\mu^* = \frac{c_1}{a_1} = \max_j\lrbr{ \frac{c_j}{a_j}\colon j\in\irg{1}{n} }$ whenever $b\in(0,a_1)$. We will now show that this gives an upper bound for all other values $b\geq a_1$. The reasoning is that the dual price of a resource can only decrease when it becomes more abundant. To see this, consider two generic duals $\min_{(\mu;\bar{\mmu})\geq \oo}\lrbr{b_i \mu + \bar{\b}\T \bar{\mmu}\colon \aa\mu+\Bb\bar{\mmu} \geq \cc }$ for $i =1,2$ with $b_1<b_2$, and their respective optimal solutions $(\mu_1^*,\bar{\mmu}_1^*)$ and $(\mu_2^*,\bar{\mmu}_2^*)$. Assume that $\mu_2^*>\mu_1^*$. Since both solutions are feasible for either problem we have $ b_2\mu_1^* + \bar{\b}\T \bar{\mmu}_1^*\geq b_2\mu_2^* + \bar{\b}\T\bar{\mmu}_2^*$ which implies
		$$\bar{\b}\T(\bar{\mmu}_1^*-\bar{\mmu}_2^*)\geq b_2(\mu_2^*-\mu_1^*)> b_1(\mu_2^*-\mu_1^*)$$
		by $\mu_2^*>\mu_1^*$ and $b_1<b_2$. This yields the final contradiction $b_1\mu_1^* + \bar{\b}\T \bar{\mmu}_1^* > b_1\mu_2^* + \bar{\b}\T\bar{\mmu}_2^*\,$, which establishes the claim.
	\end{proof}
	
	\subsection{Predictability}
	In contrast to the shortest path problem, we now consider predictability not with respect to the objective function value, but with respect to resource consumption, so we want to optimize the difference between the maximum and the minimum resource consumption of our solution, across all realizations of the uncertainty parameter. It is clear that \cref{thm:Predictability} is general enough as to cover the knapsack problem with endogenously uncertainty constraints. Still, we want to give a brief presentation of the derivation and the final result for the special case for the readers' convenience.
	\begin{align*}
		\begin{split}
			&\max_{\x,s,\ol{q},\underline{q}} \left\{ \gc_s + \ga \cc\T\x - \beta(\ol{q},\underline{q})\right \} \\
			\mathrm{s.t.:}\
			&\ol{q}  \geq \aa\T\x + \max_{\u}\lrbr{ \sum_{i=1}^{n} u_i\bar{a}^s_ix_i\colon \sum_{i=1}^{n}u_i  = \Gamma,\ \u\in[0,1]^n }, \\
			&\underline{q}  \leq \aa\T\x + \min_{\u}\lrbr{ \sum_{i=1}^{n} u_i\bar{a}^s_ix_i\colon \sum_{i=1}^{n}u_i  = \Gamma,\ \u\in[0,1]^n } ,\\
			&\x  \in \RR^s_{opt}\, ,\ s \in \irg{1}{S}\,, \left(\cc\T\x\leq \varphi^*+\eps\right) .
		\end{split}
	\end{align*}
	Again making use of the equality
	\begin{align*}
		\max \lrbr{\g\T\z:\ \e\T\z = \Gamma,\ \z\leq\e,\ \z\in\R^n_+ } = \min\lrbr{\gl \Gamma+ \ggd\T\e \colon \gl \e + \ggd \geq \g,\ \gl\in\R, \ \ggd\geq \oo}
	\end{align*}
	we arrive at
	\begin{align*}
		\begin{split}
			\max_{\x,s,\ol{q},\underline{q},\ol{\gl},\underline{\gl},\ol{\ggd},\underline{\ggd}} 			&\left\{
			\gc_s + \ga\cc\T\x - \beta(\ol{q},\underline{q}) \right\} \\
			\mathrm{s.t.:}\
			\ol{q}  &\geq \aa\T\x + \ol{\gl}\Gamma+ \ol{\ggd}\T\e\, , \ \ol{\gl}+ \ol{\gd}_i \geq \bar{a}_1^sx_i\, ,  \quad i \in \irg{1}{n}\,, \\
			\underline{q}  &\leq \aa\T\x + \underline{\gl}\Gamma+ \underline{\ggd}\T\e \,,\  \underline{\gl}+ \underline{\gd}_i \leq \bar{a}_1^sx_i \, , \quad i \in \irg{1}{n}\, ,\\
			\x  &\in \RR^s_{opt}\, ,\ s \in \irg{1}{S}\, ,\ \ol{\ggd}\geq \oo,\ \underline{\ggd}\leq \oo,\  \underline{\gl},\ol{\gl} \in \R\, , \left(\cc\T\x\leq \varphi^*+\eps\right)\, ,
		\end{split}
	\end{align*}
	which, besides the linearization of $\bar{a}_1^sx_i$ we discussed previously, needs no further reformulation.
	
	\subsection{Best-case performance}
	Again it is the best-case performance with respect to resource consumption rather than objective function value, as only the former and not the latter depends on the uncertainty parameter. The model we want to solve is thus
	\begin{align*}
		\begin{split}
			&\max_{\x,s,\ol{q}}  \left\{ \gc_s + \ga \cc\T\x + \beta(\ol{q})\right \} \\
			\mathrm{s.t.:}\
			&\ol{q}  \leq \aa\T\x + \max_{\u}\lrbr{b-\sum_{i=1}^{n} u_i\bar{a}^s_ix_i\colon \sum_{i=1}^{n}u_i  = \Gamma,\ \u\in[0,1]^n }, \\
			&\x  \in \RR^s_{opt}\, ,\ s \in \irg{1}{S}\, , \left(\cc\T\x\leq \varphi^*+\eps\right)\, ,
		\end{split}
	\end{align*}
	where the inner max-operator can be dropped due to the direction of the inequality, and the multilinear terms $u_i\bar{a}^s_ix_i$ can be exactly realized using McCormick envelopes where the arising big-M constants are all equal to 1.	
	
	\subsection{Numerical experiments: bilevel model for uncertain knapsack problems}
	For the uncertain knapsack problem we set up an experiment analogous to the SPP setting and generated instances of this problem with $n =100,\ \Gamma = 10,\ b=12$ and $S = 10$. The vector of utilities $\cc$ and the vector of of nominal weights $\aa$ were created randomly, where each entry was drawn from the interval $[1,2]$. The entries of the vectors of perturbation $\bar{\aa}_s$, $ i \in\irg{1}{S}$, were uniformly drawn from $[1,5]$.
	
	We want to stress that we chose above parameters for generating the instances deliberately, to ensure strong effects of different uncertainty preferences in the objective for a small number of uncertainty regimes.
	In any case it should be clear that results from random instances can not be generalized to real world instances in a straightforward way, in particular when $S$ is larger. We will shortly discuss which conclusions can be drawn safely from our experiments.
	
	For the formulation of the bilevel optimization problem we followed the discussion in \cref{apx:The knapsack problem with endogenous Gamma-uncertainty}, i.e. we used \cref{thm:BinrayRef1} for the characterization of $\RR_{opt}^s$, where we used lazy cuts in order to add the respective constraints as necessary. The rest of the setup was analogous to the experiments in the preceding section.
	
	\begin{figure}[h!]
		\centering
		\includegraphics[scale= 0.40]{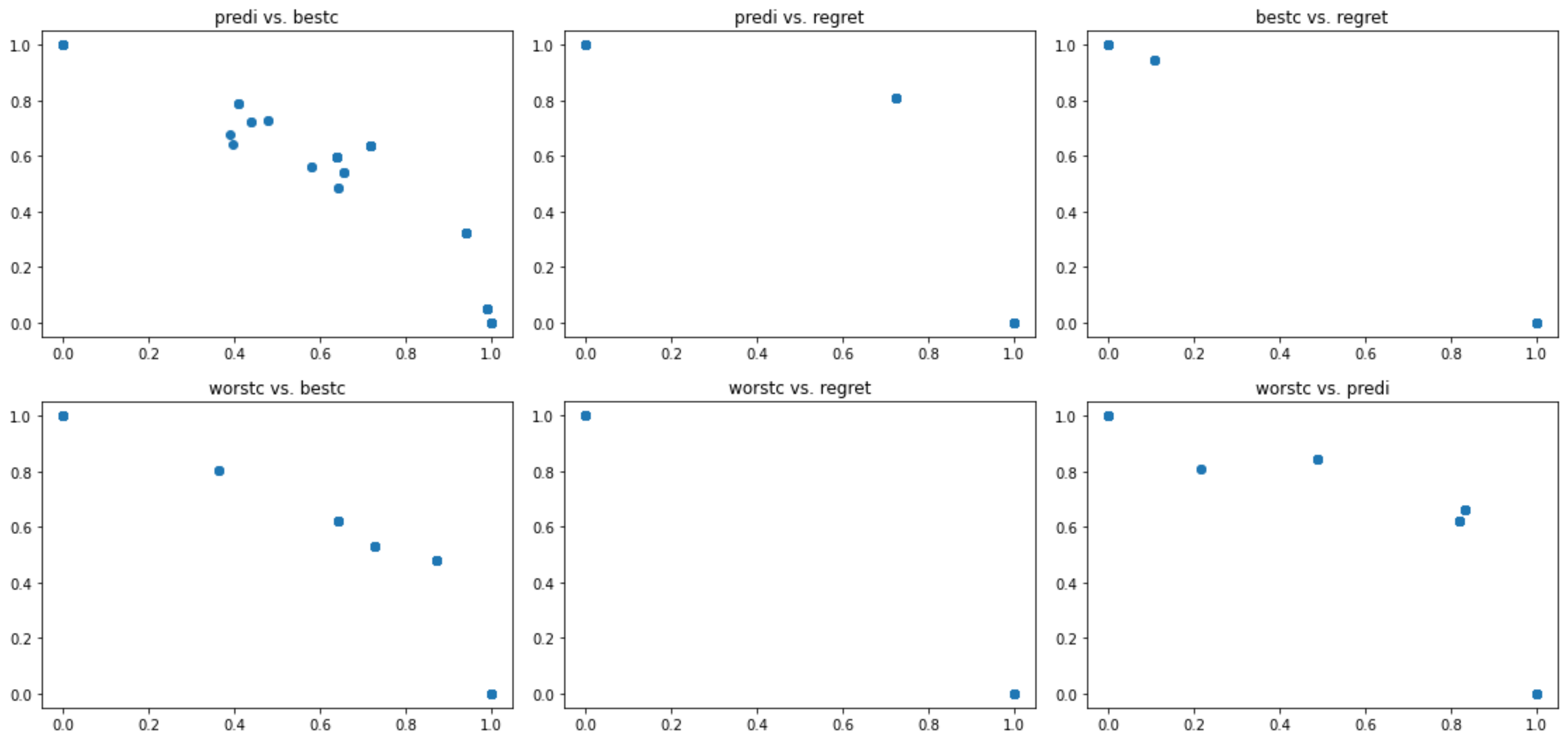}	
		\caption{Scatter plots for the knapsack problem with $\Gamma$-uncertainty}
		\label{fig:KnapSackScatter}	
	\end{figure}
	Again we see from the scatter plots in \cref{fig:KnapSackScatter} that, at least in our instances, when optimizing over $\RR_{opt}^s$ there is a trade-off to be made between the different attributes we were considering. Thus, in a bilevel-setting, a leader may exert meaningful influence over the decision of the second-stage robust optimizer.  Of course it is impossible to infer whether this phenomenon exists in all real-world instances as well,
	but our experiments confirm that there is a potential for this approach.
	\begin{table}[h!]
		\centering
		\begin{tabular}{c| c c c c c }
			preferences& $T$  & [$T_{min}$, $T_{max}$]  & $\#TO$  & $\Delta_1$\%    & $\Delta_2$ \% \\
			\hline
			\hline
			worstc vs predi   &21.56 	&[6.42, 83.26]	&16	 &7.73	 &7.12 \\
			
			worstc vs. bestc  &11.87	&[2.63, 29.2]	&19	 &10.34	 &11.11\\
			
			worstc vs. regret &7.27		&[2.63, 39.83]	&12	 &4.30   &8.45\\
			
			predi vs. bestc   &25.88	&[8.84, 105.38]	&20	 &23.41  &18.38\\
			
			predi vs. regret  &6.82		&[2.23, 30.49]	&8	 &7.11	 &5.77 \\
			
			bestc vs. regret  &5.34 	&[1.44, 23.27]	&8	 &4.22	 &4.97 \\
		\end{tabular}
		\caption{ Results for knapsack instances with interval uncertainty}\label{tbl:Knapsack}	
	\end{table}
	
	Similar conclusions can be drawn from \cref{tbl:Knapsack}, where we observe varying degrees of trade-off between the different attributes. We want to highlight that we were able to solve the instances in seconds.
	
	
	\section{Appendix: proof of \cref{thm:Predictability}}\label{apx:proof of Theorem6}
	
	\begin{proof}[Proof of \cref{thm:Predictability}]
		The proof follows immediately from standard arguments based on convex duality that are well known in robust optimization. For the readers' convenience we present the key idea in the notation used here. First, by invoking convex duality we have
		\begin{align*}
			\sup_{\u\in\UU}\lrbr{\d\T\u} =\sup_{\u}\lrbr{\d\T\u \colon \left(1,\u\T \right)\T\in\KK} =
			\inf_{\rho}\lrbr{\rho\colon \left(\rho,-\d\T\right)\T\in\KK^*}.
		\end{align*}
		Here strong duality holds, since we assumed that the uncertainty sets have nonempty interior and thus Slater's condition holds.
		We thus have the equivalence
		\begin{align*}
			\sup_{\u \in \UU}\lrbr{\d\T\u + d_0 }\leq 0 \ \Leftrightarrow \ d_0 + \rho \leq 0\, \mbox{ and }\, \left(\rho,-\d\right)\T\in\KK^*\, .
		\end{align*}		
		The reformulation follows by using this equivalence with $\d$ replaced  by $\x\T\Ab_i-\b_i\T$ and $\x\T\Cb$ respectively. For the infima the analogous procedure holds. Also, we eliminated the epi- and hypographical variables $\ol{q}_i$ and
		$\uline{q}_i\, ,\mbox{ all }i\in\irg{0}{m}$.
	\end{proof}

	\section{Appendix: open questions and outlook}\label{apx:open questions and outlook}
	
	Despite the exhaustive discussion of our framework, we feel that there are still many interesting open questions left. We wish to use this section to give a brief account of what we consider to be interesting directions for future research.
	
	In our discussions on different choices for $\MM(\x,s)$, we introduced convex functions $\beta(.)$ and $\beta_i(.)$, but settled for very simple choices of those functions in our applications. Obviously other choices would be viable from a computational standpoint, but the question is whether there are meaningful choices that can be justified based upon the decision makers' preferences.
	
	Throughout the paper we assumed that the uncertainty set contains the realizations of the uncertainty parameter for which the decision maker takes responsibility. However, this is not always how we construct uncertainty sets. In literature we find constructions of uncertainty sets that are aimed at at providing an equivalent reformulation for optimization of a risk measure. It does not to make sense to optimize predictability, best-case performance or max-regret over such uncertainty sets. However, the overall philosophy of using the flexibility of the endogenously uncertain setting in order to optimize secondary characteristics related to the uncertainty process might still be applied in these cases, but some extra effort must be made in order to transform assumptions on the uncertainty into meaningful secondary attributes. What these attributes might be and, how they can and should influence the decision is a question that in our opinion is an interesting direction for future research.
	
	In the bilevel optimization model we assumed a cooperative relationship between the leader and the follower, however, in practice this might not be necessarily the case. An alternative approach would be to assume an adversarial relationship where the follower chooses the robustly optimal solution that works against the preference of the leader. We deem such a model an interesting topic for future research.
	
	Some of the appeal of robust optimization stems from the fact that many robust counterparts can be solved using standard optimization solvers. One of our goals was to show that this is at least partly true for our framework. However, in order to have broader range of possible applications and also to have better performance of  the solution process, it would be worthwhile to investigate specialized algorithms that are able to solve (\ref{eqn:UPMModel}) in specified settings.
	
	We hope we can address these points in future research.

\end{document}